\RequirePackage{fix-cm}
\documentclass[smallextended]{svjour3}       
\smartqed  

\usepackage[english]{babel}    
\usepackage{makeidx}           
\usepackage{amsmath}           
\usepackage{amssymb}           
\usepackage[utf8]{inputenc}    
\usepackage[T1]{fontenc}       
\usepackage{lmodern}
\usepackage{eucal}             
\usepackage{bbm}               
\usepackage{longtable}         
\usepackage{exscale}           
\usepackage{mathtools}         
\usepackage[final]{graphicx}   
\usepackage[sort]{cite}        
\usepackage{mathrsfs}          
\usepackage{array}             
\usepackage{stmaryrd}          
\usepackage{paralist}          
\usepackage{wasysym}           
\usepackage{color}             
\usepackage[multiuser]{fixme}  
\usepackage{xspace}            
\usepackage{tikz}              
\usepackage[nottoc]{tocbibind} 
\usepackage[backref=page,final=true,plainpages=false,pdfpagelabels]{hyperref}   
\usepackage[]{algorithm2e}
\usepackage{rotating}
\graphicspath{{tikz/}}

\smartqed

\usetikzlibrary{matrix}
\usetikzlibrary{arrows}
\usetikzlibrary{decorations.pathreplacing}
\usetikzlibrary{calc}

\fxusetheme{color}
\FXRegisterAuthor{bs}{anbs}{bastian}

\newcommand{\I}              {\mathrm{i}}
\newcommand{\E}              {\mathrm{e}}
\newcommand{\D}              {\mathop{}\!\mathrm{d}}

\newcommand{\at}[1]          {\big|_{#1}}

\newcommand{\argument}       {\,\cdot\,}

\newcommand{\pr}             {\mathrm{pr}}

\newcommand{\functor}[1]     {\mathsf{#1}}

\DeclareMathOperator{\diag}  {\mathrm{diag}}

\DeclareMathOperator{\spann} {\mathrm{span}}

\newcommand{\DFT}{\mathsf{DFT}}

\newcommand{\Ideal}{\functor{I}}
\newcommand{\Variety}{\functor{V}}

\newcommand{\Fourier}{\mathcal{F}}

\DeclarePairedDelimiter{\ideal}{\langle \;}{{\;} \rangle}

\DeclarePairedDelimiter{\dualpair}{(}{)}

\newcommand\reallywidecheck[1]{%
\savestack{\tmpbox}{\stretchto{%
  \scaleto{%
    \scalerel*[\widthof{\ensuremath{#1}}]{\kern-.6pt\bigwedge\kern-.6pt}%
    {\rule[-\textheight/2]{1ex}{\textheight}}
  }{\textheight}%
}{0.5ex}}%
\stackon[1pt]{#1}{\scalebox{-1}{\tmpbox}}%
}

\newcommand{\algebra}[1]      {\mathscr{#1}}

\newcommand{\acts}            {\mathbin{\triangleright}}

\newcommand{\roots}{R}

\newcommand{\simpleroots}{\Delta}
\DeclarePairedDelimiter{\SP} {\langle}{\rangle}

\DeclareMathOperator{\mdeg}  {{\mathrm{deg}}_{m}}

\newcommand{\polynomials}[2][]{\Pi_{#1}^{#2}}

\newcommand{\orthopolys}[2][]{\mathcal{V}_{#1}^{#2}}

\makeatletter
\let\@@mod\mod
\DeclareRobustCommand{\mod}{\@ifstar\@mods\@@mod}
\def\@mods{\mkern4mu{\operator@font mod}\mkern6mu}
\makeatother

\title{FFT and orthogonal discrete transform on weight lattices of
  semi-simple Lie groups\thanks{This work is partially supported by
    the European Regional Development Fund (ERDF).} } \subtitle{An
  algebro-geometric perspective}

\titlerunning{FFT on weight lattices}        

\author{Bastian Seifert}

\institute{B. Seifert \at
  Ansbach University of Applied Sciences,
  Center for Signal Analysis of Complex Systems,
  91510 Ansbach,
  Germany \\  
  \emph{and} \\
  Julius-Maximilians-Universität Würzburg,
  Institute for Mathematics,
  97074 Würzburg,
  Germany \\
  Tel.: +49 (0)981 203633-28\\
  \email{bastian.seifert@hs-ansbach.de}           
}

\date{Received: date / Accepted: date}

\begin{document}

\maketitle

\begin{abstract}
    We give two algebro-geometric inspired approaches to fast
    algorithms for Fourier transforms in algebraic signal processing
    theory based on polynomial algebras in several variables. One is
    based on module induction and one is based on a decomposition
    property of certain polynomials. The Gauss-Jacobi procedure for
    the derivation of orthogonal transforms is extended to the
    multivariate setting. This extension relies on a multivariate
    Christoffel-Darboux formula for orthogonal polynomials in several
    variables. As a set of application examples a general scheme for
    the derivation of fast transforms of weight lattices based on
    multivariate Chebyshev polynomials is derived. A special case of
    such transforms is considered, where one can apply the
    Gauss-Jacobi procedure.

\end{abstract}



\section{Introduction}
\label{sec:Introduction}%

The popularization of the fast Fourier transform (FFT) algorithm by
Cooley and Tukey~\cite{Cooley.Tukey:1965} paved the way to productive
applications of the discrete Fourier transform. Due to its numerous
applications the fast Fourier transform has been termed to be one of
the most important algorithms of the twentieth century. The usage of
algebra in the theory of fast Fourier transform algorithms dates back at
least until the work of Nicholson~\cite{Nicholson:1971a}. Algebraic
approaches to FFT algorithms split into two main directions: group
algebra and polynomial algebra approaches. The interpretation of the
fast Fourier transform in terms of the cyclic group $Z_n$ was
introduced in~\cite{Nicholson:1971a}. This group-based approach
allows for a generalization of FFT algorithms to non-abelian groups
as in \cite{Diaconis.Rockmore:1990a}. The polynomial algebra
approach relies on the insight that there exists an isomorphism of
algebras $\mathbb{C}[Z_n] \cong \mathbb{C}[x] \big/ \ideal{x^n - 1}$.
This approach allows to study another large class of FFT
algorithms~\cite{Beth:1984a,Heideman.Burrus:1986a,Johnson.Burrus:1985a}
relying on ideas of Nussbaumer~\cite{Nussbaumer:1982a} and
Winograd~\cite{Winograd:1979a}. The full polynomial algebra approach
was worked out in
\cite{Pueschel.Moura:2006,Pueschel.Moura:2008a,Pueschel.Moura:2008b,Pueschel.Moura:2008c}
leading to algebraic signal processing theory. This theory captures
not only the derivation of fast algorithms for different signal models
but treats the most important concepts from linear signal processing,
e.g.\ $z$-transform, signals, filters, and Fourier transform,
algebraically, as well. One main difference in algebraic signal
processing when compared to other recent approaches, like the decomposition of
semi-simple algebras using Bratelli-diagramms in
\cite{Maslen.Rockmore.Wolff:2017a}, is that in algebraic signal
processing one decomposes \emph{modules}. This is motivated by the
fact that in algebraic signal processing theory the signals are
modeled as a module over the algebra of filters. This approach
immediately leads to explicit matrix factorizations.

In algebraic signal processing theory one can identify three
approaches for the derivation of fast algorithms for Fourier
transforms of algebraic signal models based on polynomials. Even
though all three approaches are essentially based on the Chinese
remainder theorem and a stepwise partial decomposition, the different
details lead to algorithms of different complexity.

The first one is based on a factorization of polynomials
$f(x) = g(x) \cdot h(x)$. This approach requires no special conditions
on the polynomial $f$ but leads to sub-optimal $O(n \log^2 n)$
algorithms.

The second approach is based on the decomposition property
$f(x) = p(q(x))$ of certain polynomials. This approach gives optimal
$O(n \log n)$ algorithms. Unfortunately in one variable the only
families of polynomials posessing this property are, up to
affine-linear coordinate changes, the monomials $x^n$ and the
Chebyshev polynomials $T_n(x)$ \cite{Rivlin:1974a}. Hence the only
$O(n \log n)$ algorithms for signal models based on polynomials in one
variable derivable by this method are the Cooley-Tukey-type
algorithms for the trigonometric, i.e. sine and cosine transforms
associated to Chebyshev polynomials, and the discrete Fourier
transform, associated to the monomials. In several variables it is not
known~\cite{Vesolov:1991a} if there are, up to affine-linear
coordinate changes, any examples of polynomials with this property
except the monomials and multivariate Chebyshev polynomials. The
second approach in combination with multivariate Chebyshev polynomials
was used to derive fast algorithms for undirected
hexagonal~\cite{Pueschel.Roetteler:2008} and FCC
lattices~\cite{Seifert.Hueper.Uhl:2018a}.

As there are other algorithms for the discrete Fourier transform, like
the Britanak-Rao-FFT~\cite{Britanak.Rao:1999a}, one might wonder if
these algorithms can be derived using algebraic signal processing
theory. This question was solved using the third approach. Here one
relies on induced modules. This approach raises the level of
abstraction by not relying on properties of the polynomials but on
properties of the signal modules. Module induction is based on an
algebra $\algebra{A}$ with subalgebra $\algebra{B}$ and a finite set
$T \subset \algebra{A}$, the transversal, such that
$\algebra{A} = \bigoplus_{t \in T} t \algebra{B}$. The induced
$\algebra{A}$-module $N$ of a $\algebra{B}$-module $M$ is
$N = \bigoplus_{t \in T} t \acts M$, where $\acts$ denotes the action
of the algebra on the module. In
\cite{Sandryhaila.Kovacevic.Pueschel:2011} this approach was worked
out for polynomial algebras in one variable with regular modules. As
applications general-radix algorithms for the
Britanak-Rao~\cite{Britanak.Rao:1999a} and the
Wang-FFT~\cite{Wang:1984a} were deduced.

The first part of this article extends the third and second approach
to polynomial algebras in several variables. For this a very general
decomposition of modules is used to derive a decomposition of the
Fourier marices. This is necessary since unlike in the univariate case
for multivariate polynomials the decomposition property in general
does not yield an induction. We given an interpretation of the
underlying mechanisms in the language of algebraic geometry, as well.
This is useful since it clarifies many aspects of the theory. As
an application example it is shown how one can derive the FFT for a
directed hexagonal lattice in this setting. By deriving this fast
algorithm for the directed hexagonal lattice it is illustrated how to
rederive the fast algorithms of Mersereau and
Speake~\cite{Mersereau.Speake:1981a} for regular directed lattices
within algebraic signal processing theory.

The connection of orthogonal polynomial transforms and univariate
orthogonal polynomials is well-known~\cite{Yemini.Pearl:1979a}. Using
the Christoffel-Darboux formula the Gauß-Jacobi-procedurce
of~\cite{Yemini.Pearl:1979a} allows to derive an orthogonal version of
any discrete polynomial transform based on orthogonal polynomials.
Even though there is a rather mature theory of orthogonal polynomials
in several variables~\cite{Xu:2017a} the connection to signal
processing is not vivid in the literature. Only recently the author
used the multivariate Christoffel-Darboux formula of
Xu~\cite{Xu:1993a} to derive an orthogonal version of a discrete
cosine transform on lattices of triangles~\cite{Seifert.Hueper:2018b}.
Unfortunately this method does not work in every case but relies on
the same condition as the existence of a Gaussian cubature formula as
zeros of the orthogonal polynomials. This is deplorable since Gaussian
cubatures rarely exist~\cite{Li.Xu:2010}. In the second part we derive
the general orthogonalization scheme and show that the existence of
such a Gaussian cubature implies the existence of an orthogonal
discrete transform for the signal model corresponding to the
polynomials used to construct the cubature.

Fast transforms for regular undirected lattices have been derived for
the hexagonal lattice~\cite{Pueschel.Roetteler:2008} and for the FCC
lattice~\cite{Seifert.Hueper.Uhl:2018a}. Both algorithms are special
cases of a whole family, based on generalization of the Chebyshev
polynomials to multivariate polynomials intimately connected to Lie
theory. One of the first attempts to study these polynomials in two
variables was~\cite{Koornwinder:1974}. The multivariate version was
first defined in \cite{Lidl:1975}. Important properties were deduced
in \cite{Eier.Lidl:1982}. The semigroup property was first proven in
\cite{Ricci:1986}. None of these approaches realised the connection to
Lie groups, which was clarified in \cite{Hoffman.Withers:1988a}.
Although Chebyshev polynomials in one variable are ubiquitous in
applied mathematics, their multivariate counterparts only recently
started to penetrate into applications. Meanwhile, there are
applications to the discretization of partial differential
equations~\cite{Munthe-Kaas:2006,Ryland.Munthe-Kaas:2011,Munthe-Kaas.Nome.Ryland:2012},
cubature
formulas~\cite{Li.Xu:2010,Moody.Patera:2011,Hrivnak.Motlochova.Patera:2016}
and discrete
transforms~\cite{Atoyan.Patera:2007,Hrivnak.Motlochova:2014,Hrivnak.Motlochova:2018a}.
From an algebraic signal processing perspective they are interesting
for two reasons. First they are examples of multivariate polynomials
with the decomposition property. Thus the multivariate Chebyshev
polynomials yield application examples of the generalized second
approach to fast Fourier transforms. Second they are intimately
connected to weight lattices of semi-simple Lie groups. As some of the
weight lattices are associated to densest sphere
packings~\cite[Ch.~4]{Conway.Sloane:1999}, the multivariate Chebyshev
polynomials give in these cases rise to fast transforms of optimally
sampled
signals~\cite{Petersen.Middleton:1962,Kuensch.Agrell.Hamprecht:2005}.

We derive fast algorithms in the cases of Chebyshev polynomials
associated to Lie algebras of type $A_2$ and $C_2$. Furthermore we
show that in the case $C_2$ the multivariate Gauss-Jacobi procedure is
applicable.

The main contributions of this paper are as follows. In
Sect.~\ref{sec:CooleyTukey} the induction-based
approach~\cite{Sandryhaila.Kovacevic.Pueschel:2011} and the approach
relying on the decomposition property~\cite{Pueschel.Roetteler:2008}
for the derivation fast algorithms in algebraic signal processing
theory are extended to a more general situation and polynomials in
several variables. Furthermore it is shown that in the multivariate
case the decomposition property yields another decomposition theorem
for Fourier transforms. In Sect.~\ref{sec:ChristoffelDarboux} a
generalization of the Gauß-Jacobi-procedure for the derivation of
orthogonal transforms is derived. Finally in
Sect.~\ref{sec:UndirectedLattices} we state a general scheme for the
derivation of fast transform algorithms for undirected weight lattices
of semi-simple Lie algebras based on multivariate Chebyshev
polynomials.

\section{An algebro-geometric perspective on signal processing and
  FFT} 
\label{sec:CooleyTukey}%

Algebraic signal processing theory enlightens the algebraic structures
underlying linear signal processing
technqiues~\cite{Pueschel.Moura:2006,Pueschel.Moura:2008a,Pueschel.Moura:2008b}.
From a signal processing perspective the algebraic structures can be
motivated as follows. If one considers the basic operations on
filters, i.e. putting them in series and parallel and amplifying them,
one can interpret these operations as addition, multiplication and
scalar multiplication, respectively. One observes that these
operations are subject to a distributive law. Consequently the filters can
mathematically be described by an algebra with respect to these
operations. Furthermore, one can add and amplify signals, and we can
apply filters to signals. From a mathematical point of view one thus
gets the structure of a module over the algebra of filters for the
signals, with application of filters to signals as algebra action. The
$z$-transform is a bijective mapping from a set of numbers, the
samples, to signals, which embody more structure. In this way, the
$z$-transform tells how to translate samples to signals.

An \emph{algebraic signal model}, a triple $(\algebra{A}, M, \Phi)$
consisting of a $\mathbb{C}$-algebra $\algebra{A}$, a free
$\algebra{A}$-module $M$, and a bijective mapping
$\Phi \colon \mathbb{C}^n \to M$ for some
$n \in \mathbb{N} \cup \{\infty\}$. By the previous considerations
this is motivated as the main object to study in algebraic signal
processing theory.

We recall the first example of algebraic signal processing theory, the
classical finite time discrete signal processing. In the classical
theory one considers a set of numbers
$s = (s_0,\dots,s_{n-1}) \in \mathbb{C}^n$ as signal and extends it
periodically, i.e. $s_N = s_{N \mod* n}$ for any $N \in \mathbb{Z}$. A
set of samples $s \in \mathbb{C}^n$ is mapped to a polynomial in
$x = z^{-1}$ by the $z$-transform
\begin{equation}
    \label{eq:DSPExample}
    \Phi_d \colon (s_0, \dots, s_{n-1}) \mapsto \sum_{i=0}^{n-1} s_i x^{i}.
\end{equation}
To capture the periodic extension of the signal, one considers the
resulting polynomials modulo $x^n - 1$ or, more precisely, modulo the
ideal generated by $x^n - 1$. This results in an element of the set
$\mathbb{C}[x] \big/ \ideal{x^{n} - 1}$. The filters in classical
signal processing are generated by a shift. A shift is realized on the
polynomials as multiplication by $x$
\begin{equation}
    \label{eq:ShiftInDSP}
    x \cdot \Phi(s)
    = x \sum_{i=0}^{n-1} s_i x^i
    = \sum_{i=0}^{n-1} s_{i-1 \mod* n} x^i.
\end{equation}
This results in a delay of the signal. The filters are the polynomials
in the shift $x$, i.e. elements of
$\mathbb{C}[x] \big/ \ideal{x^n - 1}$. The structural difference
between signal and filters is the algebraic structure. The set of
filters is equipped with the structure of an algebra, while the
signals form a module over the algebra of filters. In this example we
thus get the following structures. Let
$h_1, h_2 \in \mathbb{C}[x] \big/ \ideal{x^n - 1}$ be two filters and
$c \in \mathbb{C}$, then $h_1(x) \cdot h_2(x)$, $h_1(x) + h_2(x)$ and
$c \cdot h_i(x)$ form new filters. For two signals
$s_1,s_2 \in \mathbb{C}[x] \big/ \ideal{x^n - 1}$ only
$s_1(x) + s_2(x)$ and $c \cdot s_1(x)$ form new signals, i.e. one has
the structure of a vector space. But additionally
$h_i(x) \cdot s_i(x)$ forms a new filter, as well, which turns the
signals into a module over the algebra of filters. Hence the finite
time discrete signal processing translates to the model
$(\algebra{A}_d, M_d, \Phi_d)$, with underlying sets
$\algebra{A}_d = M_d = \mathbb{C}[x] \big/ \ideal{x^n - 1}$, in
algebraic signal processing theory.

A signal model can be visualized by a graph. The visualization is
given by the following construction. First one associates to each
basis element of the module a node. An edge from one node to another
is added if the result of the action of a generator of the algebra,
i.e. a shift, on the basis element associated to the first node
contains non-zero coefficient to the basis element of the second node.
The visualization of the discrete finite time model is shown in
Fig.~\ref{fig:FiniteDiscreteTimeModel}.
\begin{figure}
    \centering
    \includegraphics[width=0.8\textwidth]{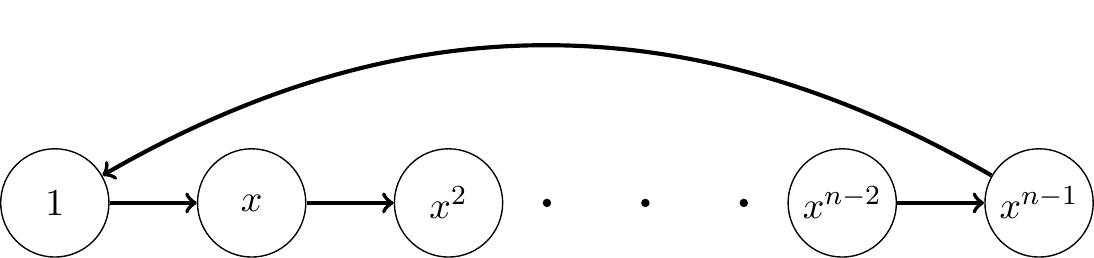}
    \caption{Visualization of the finite, discrete time signal model.}
    \label{fig:FiniteDiscreteTimeModel}
\end{figure}
Since the boundary conditions of two- or three-dimensional models
tend to lead to confusing pictures, the boundary connections are often
omitted. Note that these visualization graphs motivated a path to
signal processing on graphs see e.g.\ \cite{Sandryhaila.Moura:2013a}.

An algebraic signal model gives rise to a notion of Fourier transform
based on the decomposition of the signal module. Assume
$M \cong \bigoplus_i M_i$ can be decomposed into irreducible modules.
Any isomorphism
\begin{equation}
    \label{eq:FourierTransform}
    \Fourier \colon M \to \bigoplus_i M_i
\end{equation}
is called a \emph{Fourier transform} for the associated algebraic
signal model. For the finite time discrete signal processing model one
has the decomposition
\begin{equation}
    \label{eq:FourierTransformDiscreteFiniteTimeSignals}
    \mathbb{C}[x] \big/ \ideal{x^N - 1}
    \cong \bigoplus_{j=0}^{N-1} \mathbb{C}[x] \big/ \ideal{x - \E^{2
        \pi \I j/N}}.
\end{equation}
Choosing as basis $[1,x,\dots,x^{N-1}]$ in
$\mathbb{C}[x] \big/ \ideal{x^N - 1}$ and $[1]$ in each
$\mathbb{C}[x] \big/ \ideal{x - \E^{2 \pi \I j/N}}$ the
isomorphism~\eqref{eq:FourierTransformDiscreteFiniteTimeSignals} can be
realized by the $N \times N$-matrix
\begin{equation}
    \label{eq:FourierTransformDiscreteFiniteTimeSignalsMatrix}
    \left[ \E^{2 \pi \I j \cdot k/N} \right]_{j,k},
\end{equation}
which is the discrete Fourier transform matrix.

Fast algorithms for these Fourier transforms rely on step-wise
application of the Chinese remainder theorem. We recall the basic
notations needed, for a more detailed treamtent see
e.g.~\cite[Ch.~II]{Lang:2002a}. Recall that ideals $I_1,I_2$ of a ring
$R$ are called coprime if $I_1 + I_2 = R$. Typical examples of coprime
ideals in the polynomial algebra $\mathbb{C}[x]$ are the ideals
$\ideal{x-a}$ and $\ideal{x-b}$ with scalars $a \not= b$. For any
commutative ring $R$ with ideal $I = I_1 \cap \dots \cap I_n$, such
that all $I_i$ are coprime, one has an isomorphism
\begin{equation}
    \label{eq:CRTIsoFancy}
    R \big/ I \longrightarrow R \big/ I_1 \times \dots \times R\big/
    I_n. 
\end{equation}
Tensoring this isomorphism with an $R$-module $M$ yields the Chinese
remainder theorem for modules
\begin{equation}
    \label{eq:CRTIsoFancyModules}
    M \big/ I M \longrightarrow M \big/ I_1 M \times \dots \times M
    \big/ I_n M. 
\end{equation}

We are especially interested in algebras of polynomials in several
variables. This is motivated by two considerations. First, group
algebras and algebras of polynomials in one variable have been
investigated already thoroughly. Second, algebras of polynomials in
several variables are intimately connected to algebraic geometry. This
allows one to get a geometric point of view for signal processing
concepts. One immediate obstacle using polynomials in multiple
variables is that their zero-sets are in general non-discrete. Hence
one has to be aware that we will consider only very special cases of
polynomial algebras.

For the geometric point of view recall the Hilbert
Nullstellensatz~\cite[Ch.~IX]{Lang:2002a}, which states that there is
a correspondence between ideals of a polynomial algebra and varieties.
A variety is a subset of $\mathbb{C}^n$, being the set of common zeros
of all polynomials in the ideal. This correspondence is not one-to-one
since for example $\ideal{x}$ and $\ideal{x^2}$ have the same variety
$\Variety(\ideal{x}) = \Variety(\ideal{x^2}) = \{0\}$. But the
correspondence becomes one-to-one if one restricts to radical ideals.
An ideal $I$ of an algebra $\algebra{A}$ is called radical if
$I = \sqrt{I} = \{ a^n \; | \; a \in I, n \in \mathbb{N}\}$. For
example the radical ideal of $\ideal{x^2}$ is $\ideal{x}$. Radicality
of an ideal is the several variables analog of square-freeness of
polynomials in one variable. An ideal is called zero-dimensional if
its variety $\Variety(I)$ is finite. \textbf{In this paper we assume
  all ideals to be radical and zero-dimensional.} If this is not the
case for some example, we will explicitly state that and use the
radical ideal.

Another problem with multivariate polynomials is that in general
division by the generators of an ideal is not well-defined. That is
one can get different results of the division by changing the sequence
of which generator to divide by. The crucial notion to avoid this
problem is that of a Gröbner basis. A Gröbner basis is a special set
of generators of an ideal depending on the choice a monomial order,
see \cite{Cox.Little.OShea:2015} for details. By the Buchberger
criterion a set of polynomials forms a Gröbner basis if their leading
monomials with respect to the choosen monomial order are disjoint. In
this paper we rely on this criterion only to decide whether a given
set of generators is a Gröbner basis.

Additionally to the algebra $\algebra{A}$ we are considering an
$\algebra{A}$-module. The geometric counterpart of a module over an
algebra is a vector bundle over a space. This is formalized by the
Serre-Swan theorem, see e.g.~\cite{Morye:2013a}. The Serre-Swan
theorem states that the sections of vector bundles are precisely the
projective, finitely generated modules over the algebra of functions
of the underlying space. So the module of signals of an algebraic
signal model can be interpreted as sections of vector bundles, cf.
Fig.~\ref{fig:SectionsVectorBundleCircle}.
\begin{figure}
    \centering
    \includegraphics[width=0.5\textwidth]{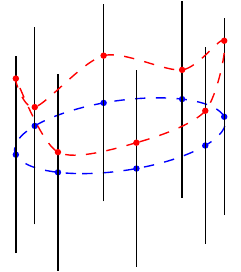}
    \caption{Sections of a vector bundle (red) over points on a circle
      (blue) form the signals of the finite, discrete time signal
      model.}
    \label{fig:SectionsVectorBundleCircle}
\end{figure}
\begin{remark}
    \label{remark:FiniteFields}%
    In principle one does not need to restrain to $\mathbb{C}$ as the
    ground field. Indeed signal processing and Fourier transforms
    using finite fields might be of interest in some applications as
    these can be used for infinite precision calculations, see
    e.g.~\cite{Lima.CampelloDeSouza:2011a}. Nonetheless we will only
    consider algebras over $\mathbb{C}$ in this paper since this
    simplifies some arguments and definitions. If all the coefficients
    and varieties of appearing polynomials are real, we consider the
    structures over $\mathbb{R}$ without any loss.
\end{remark}

We denote by $\polynomials{n}(x) = \mathbb{C}[x_1,\dots,x_n]$ the
space of all polynomials in $n$ indeterminates. Let the filter algebra
be of the form $\algebra{A} = \polynomials{n}(x) \big/ I$ for some
radical, zero-dimensional ideal $I$. Then by the Chinese remainder
theorem we have
\begin{equation}
    \label{eq:CRTFiniteVariety}
    \polynomials{n}(x) \big/ I
    \cong \bigoplus_{\alpha \in \Variety(I)}
    \polynomials{n} \big/ \ideal{x_{1} - \alpha_{1}, \dots,
      x_{n} - \alpha_{n}}, 
\end{equation}
as all the $ \ideal{x_1 - \alpha_{1}, \dots, x_n - \alpha_{n}}$ are
coprime as maximal ideals. The corresponding Fourier transform for the
signal model with regular module is realized by the map
\begin{equation}
    \label{eq:FourierTrafoFiniteVariety}
    p(x_1,\dots,x_n) \mapsto
    \begin{bmatrix}        
        p(\alpha)
    \end{bmatrix}_{\alpha \in \Variety(I)}.
\end{equation}
If we choose a basis $B$ in the module $M = \polynomials{n} \big/ I$
and the basis $[1]$ consisting of one only in each
$\polynomials{n} \big/ \ideal{x_1 - \alpha_{1}, \dots, x_n -
  \alpha_{n}}$, the Fourier transform can be realized as
multiplication with the matrix
\begin{equation}
    \label{eq:FourierTrafoFiniteVarietyMatrixForm}
    P_{b,M} = [b(\alpha)]_{b \in B, \alpha \in \Variety(I)}. 
\end{equation}
If other bases than $[1]$ are used in each
$\polynomials{n}(x) \big/ \ideal{x_1 - \alpha_{1}, \dots, x_n -
  \alpha_{n}}$, e.g. $[a_i]$ for $a_i \in \mathbb{C}$, the matrix
changes to
\begin{equation}
    \label{eq:FourierTrafoFiniteVarietyMatrixForm}
    P_{b,M} = \diag\bigg(\frac{1}{a_i}\bigg) [b(\alpha)]_{b \in B, \alpha \in
      \Variety(I)}.  
\end{equation}
A fast algorithm for the Fourier transform $P_{b,M}$ is a
factorization of the dense matrix $P_{b,M}$ into sparse matrices.

As motivation for the following we deduce the Cooley-Tukey FFT
algorithm for the $\mathsf{DFT}$, first with a
top-down~\cite{Pueschel.Moura:2008c} then with a bottom-up
approach~\cite{Sandryhaila.Kovacevic.Pueschel:2011}. This motivates
the general theory. Consider the finite, discrete time signal model.
If $n = m \cdot r$, we have $x^n = (x^m)^r$, and
$(x^m)^r - 1 = \prod_{j=1}^{r} (x^m - \E^{2 \pi \I j/r})$, as well as
$x^m - \E^{2 \pi \I j/r} = \prod_{k=1}^{m} (x - \E^{2 \pi \I j k/(m
  r)})$. Hence we can decompose the module
$\mathbb{C}[x] \big/ \ideal{x^n - 1}$ in steps using the Chinese
remainder theorem
\begin{equation}
    \label{eq:FourierDecompositionTopDown}%
    \begin{split}
        \mathbb{C}[x] \big/ \ideal{x^n - 1}
        &\longrightarrow \mathbb{C}[x] \big/ \ideal{(x^m)^r - 1}
        \\
        &\longrightarrow \bigoplus_{j=1}^{r} \mathbb{C}[x]
        \big/ \ideal{x^m - \E^{2 \pi \I j/r}} \\
        &\longrightarrow \bigoplus_{j=}^{r}
        \bigoplus_{k=1}^{m} \mathbb{C}[x] \big/ \ideal{x - \E^{2
            \pi \I j k / (r m)}} \\
        &\longrightarrow \bigoplus_{\ell=1}^{n} \mathbb{C}[x]
        \big/ \ideal{x - \E^{2 \pi \I \ell/n}}.
    \end{split}
\end{equation}
Each step is described by a sparse matrix and via the recursion step,
we obtain the well-known $O(n \log(n))$ Cooley-Tukey algorithm in the
case $n = 2^m$. This is an example of the top-down approach for the
derivation of FFT-like algorithms by algebraic signal processing
theory used in \cite{Pueschel.Moura:2008c}.

For a bottom-up approach we need to recall some tools from the
representation theory of algebras. Consider a subalgebra $\algebra{B}$
of the algebra $\algebra{A}$. A finite set $T$ is called transversal
if
\begin{equation}
    \label{eq:Transversal}
    \algebra{A} = \bigoplus_{t \in T} t \algebra{B} 
\end{equation}
as vector spaces. If $M$ is a $\algebra{B}$-module, the module
$\bigoplus_{t \in T} t \acts M$, with vector space direct sum and
$\acts$ denoting the action of the algebra on the module, is the
$T$-induced $\algebra{A}$-module of $M$.

Now consider again the discrete finite time signal model with
$n = m \cdot r$. Then the algebra
$\algebra{B} = \mathbb{C}[y] \big/ \ideal{y^m - 1}$ is a subalgebra of
$\algebra{A} = \mathbb{C}[x] \big/ \ideal{x^n - 1}$ generated by
$y = x^r$. A transversal of $\algebra{B}$ in $\algebra{A}$ is given by
$T = \{1,x,\cdots,x^{r-1}\}$. The identification of the regular module
$\mathbb{C}[x] \big/ \ideal{x^n - 1}$ with the induced module
$\bigoplus_{t \in T} t \mathbb{C}[y]\big/ \ideal{y^m - 1}$ is done by
a basis change. This results in the decomposition
\begin{equation}
    \label{eq:FourierDecompositionBottomUp}%
    \begin{split}
        \mathbb{C}[x] \big/ \ideal{x^n - 1}
        &\longrightarrow \bigoplus_{t \in T} t \mathbb{C}[y]
        \big/ \ideal{y^m - 1} \\
        &\longrightarrow \bigoplus_{t \in T} \bigoplus_{k=1}^{m} t
        \mathbb{C}[y] \big/ \ideal{y - \E^{2 \pi \I k /m}} \\ 
        &\longrightarrow \bigoplus_{\ell=1}^{n} \mathbb{C}[x]
        \big/ \ideal{x - \E^{2 \pi \I \ell/n}}.
    \end{split}
\end{equation}
Even though this decomposition resembles at a first glance the
top-down approach, the explicit matrix form shows that these
approaches are somewhat dual to each other. The top-down approach
results in a decimation-in-frequency approach while the bottom-up
yields a decimation-in-time approach. For example with
$n = 4 = 2 \cdot 2$ the decomposition
\eqref{eq:FourierDecompositionTopDown} leads to
\begin{equation}
    \label{eq:FourierMatricesTopDown}%
        \begin{bmatrix}
            1 & 1 & 1 & 1 \\
            1 & -\I & -1 & \I \\
            1 & -1 & 1 & -1 \\
            1 & \I & -1 & -\I
        \end{bmatrix} \\
        =
        \begin{bmatrix}
            1 & 0 & 0 & 0\\
            0 & 0 & 0 & 1 \\
            0 & 1 & 0 & 0\\
            0 & 0 & 1 & 0 
        \end{bmatrix}
        \begin{bmatrix}
            1 & 1 \\
            1 & -1 \\
            & & 1 & \I \\
            & & 1 & -\I
        \end{bmatrix}        
        \begin{bmatrix}
            1 & 0 & 1 & 0 \\
            0 & 1 & 0 & 1 \\
            1 & 0 & -1 & 0 \\
            0 & 1 & 0 & -1
        \end{bmatrix},
\end{equation}
while the decomposition \eqref{eq:FourierDecompositionBottomUp} gives
\begin{equation}
    \label{eq:FourierMatricesBottomUp}%
    \begin{bmatrix}
        1 & 1 & 1 & 1 \\
        1 & -\I & -1 & \I \\
        1 & -1 & 1 & -1 \\
        1 & \I & -1 & -\I
    \end{bmatrix} \\
    =
    \begin{bmatrix}
        1 & 0 & 1 & 0 \\
        0 & 1 & 0 & -\I \\
        1 & 0 & -1 & 0 \\
        0 & 1 & 0 & \I 
    \end{bmatrix}
    \begin{bmatrix}
        1 & 1 \\
        1 & -1 \\
        & & 1 & 1 \\
        & & 1 & -1
    \end{bmatrix}
    \begin{bmatrix}
        1 & 0 & 0 & 0\\
        0 & 0 & 1 & 0 \\
        0 & 1 & 0 & 0 \\
        0 & 0 & 0 & 1 
    \end{bmatrix}.
\end{equation}
Hence both approaches lead to a sparse factorization of the matrices.
Indeed in this case, as well as in general for univariate polynomials,
both approaches can be described using induced modules. This relies on
each of the modules
$\mathbb{C}[x] \big/ \ideal{x^m - \E^{2 \pi \I j/r}}$ being isomorphic.
If one has more than one variable this is not true anymore.
Furthermore, it might be advantageous to include more general modules
and module actions than the regular module.

The basic idea for fast algorithms is stated in the following
commutative diagram~\eqref{eq:BigPictureCooleyTukey}. It shows how one
can decompose the $\algebra{A}$-module $N$ stepwise if one can
represent it a as an induction.
\begin{equation}
    \label{eq:BigPictureCooleyTukey}
    \begin{tikzpicture}[baseline=(current bounding box).center]
        \matrix(m)[matrix of math nodes,
        row sep = 2.5em,
        column sep = 9em]{
          N  &  \bigoplus_{t} t \acts M_t \\          
          \bigoplus_i \polynomials{n} \big/
          \ideal{ x - \alpha_i}
          &  \bigoplus_{t} \bigoplus_j
          \polynomials{n}(y) \big/ \ideal{ y - \beta_{t,j}} \\  
        };
        \draw[->] (m-1-1) to node[midway, above, sloped]{Basis change} (m-1-2);
        \draw[->] (m-1-1) to node[right]{Fourier} (m-2-1);
        \draw[->] (m-1-2) to node[right]{Fourier} (m-2-2);
        \draw[->] (m-2-2) to node[midway, below, sloped]{Project and
          Scale} (m-2-1); 
    \end{tikzpicture}
\end{equation}
The diagram is cast into a theorem.
\begin{theorem}[FFT algorithms, bottom-up approach]
    \label{theorem:CooleyTukeyTypeAlgorithms}%
    Let $\algebra{A}$ be an algebra with subalgebra $B$. Let
    $M_t = \polynomials{d}(y) \big/ J_t$ be a set of
    $\algebra{B}$-modules such that
    $N = \bigoplus_{t \in T} t \acts M_t = \polynomials{n} \big/ I$,
    with $T = \{t_1,\dots,t_w\} \subset \algebra{A}$ a finite set, is
    an $\algebra{A}$-module. Assume the action of $t \in T$ on $M_t$
    is by multiplication with a polynomial $t$. Let
    $r_t \colon V(I) \longrightarrow V(J_t)$ be a surjective map
    between the corresponding varieties. The Fourier transform of $N$
    with respect to a basis $b_N$ can be decomposed as
    \begin{equation}
        \label{eq:CooleyTukeyTypeAlgorithms}
        P_{b_N, N} = [D_1 R_1 | \dots | D_w R_w] \left[ \bigoplus_{t
          \in T} P_{b_t, t \acts M} \right] B_{\oplus b_t}^{b_N},
    \end{equation}
    where $B_{\bigoplus b_t}^{b_N}$ is the basis change from the basis
    $b_N$ to the concatenation of bases of the $t \acts M_t$, the
    matrices $P_{b_t, t \acts M_t}$ are the Fourier transforms of the
    $t \acts M_t$, the matrices $R_t$ are matrices with entries
    $[R_t]_{\alpha \in \Variety(I), \beta \in \Variety(J)}$ being $1$
    if $r_t(\alpha) = \beta$ and $0$ otherwise, and the 
    $D_t = \diag\left( t(\alpha) | \alpha \in \Variety(I)\right)$ .
\end{theorem}
\begin{proof}
    First note that since $t$ acts as multiplication with a polynomial
    any element of $t \acts M_t$ can be written as $t \cdot b$ and
    denote the choosen basis without the $t$ as $b_t$. Then the claim
    follows from the following unwinding of definitions
    \begin{equation*}
        \begin{split}
            P_{b_N,N} &= [b(\alpha)]_{b \in b_N, \alpha \in
              \Variety(I)} \\
            &= [t(\alpha) b(r(\alpha))_{b \in \oplus b_t, \alpha \in
              \Variety(I)} | t \in T] B_{\oplus b_t}^{b_N} \\
            &= [ \diag(t(\alpha) | \alpha \in \Variety(I)) R_t
            (b(\beta))_{b \in b_t, \beta \in \Variety(J_t)}] B_{\oplus
              b_t}^{b_N}  \\
            &= [D_1 R_1 | \dots | D_w R_w] \left[ \bigoplus_{t \in T}
                P_{b_t, t \acts M_t} \right] B_{\oplus b_t}^{b_N},
        \end{split}
    \end{equation*}
    where
    $(b(r(\alpha))_{b \in b_t, \alpha \in \Variety(I)} = R_t
    (b(\beta))_{b \in b_t, \beta \in \Variety(J_t)}$ follows since
    $r \colon \Variety(I) \longrightarrow \Variety(J_t)$ is onto and
    $R_t$ keeps track of this map. The result follows. \qed
\end{proof}
\begin{example}
    \label{example:FFTrevisited}%
    Checking consistency the matrix decomposition
    \eqref{eq:FourierMatricesBottomUp} is derived using
    Theorem~\ref{theorem:CooleyTukeyTypeAlgorithms}.
    The module $\mathbb{C}[x] \big/ \ideal{x^4 - 1}$ can be
    represented as $1 \acts \mathbb{C}[y] \big/ \ideal{y^2 - 1} \oplus
    x \acts \mathbb{C}[y] \big/ \ideal{y^2 - 1}$ with transversal $T =
    \{1,x\}$. 
    The basis change
    $B$ is from $\{1,x,x^2,x^3\}$ to $\{1,x^2, x,x^3\}$ and thus
    \begin{equation*}
        B =
        \begin{bmatrix}
            1 & 0 & 0 & 0\\
            0 & 0 & 1 & 0 \\
            0 & 1 & 0 & 0\\
            0 & 0 & 0 & 1
        \end{bmatrix}.        
    \end{equation*}
    The direct sum is
    $\mathbb{C}[y] \big/ \ideal{y^2 - 1} \oplus x \mathbb{C}[y] \big/
    \ideal{y^2 - 1}$ as modules leading to the matrix
    \begin{equation*}
        \DFT_2 \oplus
        \DFT_2 =
        \begin{bmatrix}            
            1 & 1  & & \\
            1 & -1 & & \\
            & & 1 & 1 \\
            & & 1 & -1 
        \end{bmatrix}.        
    \end{equation*}
    The matrices $R_{1}$ and $R_2$ keeping track of the map between
    the varieties are given by
    \begin{equation*}
        M_i =
        \begin{bmatrix}
            1 & 0 \\
            0 & 1 \\
            1 & 0 \\
            0 & 1
        \end{bmatrix},        
    \end{equation*} since $r = x^2$ maps
    $\Variety(\ideal{x^4 - 1}) = \{1,\I,-1,-\I\}$ onto
    $\{1,-1,1,-1\}$. The diagonal matrix $D_1$ is the identity since
    the polynomial $1$ evaluates always to $1$, while
    \begin{equation*}
        D_2 =
        \begin{bmatrix}
            1 \\
            & \I \\
            & & -1 \\
            & & & -\I
        \end{bmatrix}.        
    \end{equation*}
    Hence we obtain
    \begin{equation*}
        [D_1 M_1 \; | \; D_2 M_2 ] =
        \left[ \begin{array}{c c | c c}
            1 & 0 & 1 & 0 \\
            0 & 1 & 0 & \I \\
            1 & 0 & -1 & 0 \\
            0 & 1 & 0 & -\I
        \end{array}\right] .
    \end{equation*}
    The matrix decomposition obtained using
    theorem~\ref{theorem:CooleyTukeyTypeAlgorithms} hence coincides
    with the decomposition~\eqref{eq:FourierMatricesBottomUp}. 
\end{example}
We want to investigate, how we can ensure existence of a transversal.
We start by characterizing subalgebras generated by exactly the number
of variables generators. This is done in terms of the image of the
variety under the image of the generators of the subalgebra.
\begin{proposition}
    \label{proposition:CharacterizationOfPolynomialSubalgebras}%
    Let $\algebra{B} \subseteq \algebra{A} = \polynomials{n} \big/ I$
    be a finitely generated subalgebra, s.t.
    $\algebra{B} = \ideal{r_1, \dots, r_n}$ for $r_i \in \algebra{A}$.
    Then as algebras
    \begin{equation}
        \label{eq:CharacterizationOfPolynomialSubalgebras}
        \algebra{B} \cong  \polynomials{n}(y) \big/ J,
    \end{equation}
    where $J = \Ideal((r_1,\dots,r_n)(\Variety(I)))$ is the ideal of
    the image of $\Variety(I)$ under the generators of $\algebra{B}$
    in $\mathbb{C}[y_1,\dots,y_n]$.
\end{proposition}
\begin{proof}
    To proof \eqref{eq:CharacterizationOfPolynomialSubalgebras}, we
    show that both sides have the same dimension and the kernel of an
    algebra homomorphism between them is trivial.

    Denote the finite variety by $\{\alpha_1,\dots,\alpha_k\} = \Variety(I)$.
    Let $\{\beta_1,\dots,\beta_\ell\}$ the image of these points under
    $(r_1,\dots,r_n)$. Then $\ell \leq k$.
    \begin{claim}[1]
        \label{claim:DimensionOfSubalgebra}%
        $\dim \algebra{B} = \ell$.
    \end{claim}
    We prove Claim~(1). We can write
    \begin{equation*}
        I = \prod_i \ideal{x_1 - \alpha_{i,1}, \dots, x_n - \alpha_{i,n}}.
    \end{equation*}
    Each of the
    $\ideal{x_1 - \alpha_{i,1}, \dots, x_n - \alpha_{i,n}}$ is
    maximal, hence they are all coprime and we can use the Chinese
    remainder theorem~\eqref{eq:CRTFiniteVariety} to decompose
    $\algebra{A}$. Denote by
    \begin{equation*}
        \Fourier \colon \algebra{A} \longrightarrow \bigoplus_i
        \polynomials{n} \big/ \ideal{x_1 - \alpha_{i,1}, \dots,
          x_n - \alpha_{i,n}}   
    \end{equation*}
    the isomorphism from
    equation~\eqref{eq:FourierTrafoFiniteVarietyMatrixForm}. The
    diagram
    \begin{equation*}
        \begin{tikzpicture}[baseline=(current bounding box).center]
            \matrix(m)[matrix of math nodes,
            row sep = 2em,
            column sep = 4em
            ]{
              \ker(\pr_{\algebra{B}}) & \ker(\pr) \\
              \algebra{A} & \bigoplus_i \polynomials{n}
              \big/ \ideal{x_1 - \alpha_{i,1}, \dots, x_n - \alpha_{i,n}} \\
              \algebra{B} & \Fourier(\algebra{B}) \\
              0 & 0 \\
            };
            \draw[->] (m-1-1) to (m-2-1);
            \draw[->] (m-2-1) to node[left]{$\pr_{\algebra{B}}$} (m-3-1);
            \draw[->] (m-3-1) to (m-4-1);
            
            \draw[->] (m-1-2) to (m-2-2);
            \draw[->] (m-2-2) to node[right]{$\pr$} (m-3-2);
            \draw[->] (m-3-2) to (m-4-2);

            \draw[->] (m-2-1) to node[above]{$\Fourier$} (m-2-2);
            \draw[->] (m-3-1) to node[above]{$\Fourier$} (m-3-2); 
        \end{tikzpicture}            
    \end{equation*}
    commutes. Hence it suffices to determine the dimension of
    $\ker(\pr)$, to determine the dimension of $\algebra{B}$. But the
    dimension of $\ker(\pr)$ is given by the number of $\alpha_i$,
    which get mapped to the same $\beta_j$ under the $r_i$, so
    $\dim \ker(\pr) = k - \ell$. Henceforth
    $\dim \algebra{B} = \dim \Fourier(\algebra{B}) = k - \dim
    \ker(\pr) = \ell$. Hence
    $\dim \algebra{B} = \dim \polynomials{n}(y) \big/ J$, as $J$ is
    radical and $|\Variety(J)| = \ell$. This proves claim~(1).

    Consider the algebra homomorphism
    \begin{equation*}
        \begin{split}
            \kappa \colon & \algebra{B} \longrightarrow
            \polynomials{n}(y) \big/ J \\
            & r_i \mapsto y_i,
        \end{split}
    \end{equation*}
    which maps generators to generators. We have the short exact
    sequence
    \begin{equation*}
        \begin{tikzpicture}[baseline=(current bounding box).center]
            \matrix(m)[matrix of math nodes,
            row sep = 2em,
            column sep = 4em]{
              0 & \kappa^{-1}(J) & \algebra{B} &
              \polynomials{n}(y) \big/ J & 0 \\
            };
            \draw[->] (m-1-1) to (m-1-2);
            \draw[->] (m-1-2) to (m-1-3);
            \draw[->] (m-1-3) to node[above]{$\kappa$} (m-1-4);
            \draw[->] (m-1-4) to (m-1-5);
        \end{tikzpicture}
    \end{equation*}
    and hence
    $ \polynomials{n}(y) \big/ J \cong \algebra{B} \big/
    \kappa^{-1}(J)$. So we still need to show:
    \begin{claim}[2]
        \label{claim:KernelOfAlgebraHomomorphism}%
        $\ker(\kappa) = \{0\}$.
    \end{claim}
    We prove Claim~(2). It suffices to show, that the $r_i$ vanish on
    the ideal $\ker(\kappa) = \kappa^{-1}(J)$. As $J$ is the ideal of
    the points $\beta_i$, it can be written as
    \begin{equation*}
        J = \prod_j \ideal{y_1 - \beta_{j,1}, \dots, y_n - \beta_{j,n}}.
    \end{equation*}
    So
    $\kappa^{-1}(J) = \prod_i \ideal{r_1 - \beta_{i,1}, \dots, r_n -
      \beta_{i,n}}$. Now the isomorphism $\Fourier$ maps the $r_i$ to
    the $\beta_i$, as the $\beta_i$ are the image of the $\alpha_i$
    under $r_i$ and $\Fourier$ is, by
    \eqref{eq:FourierTrafoFiniteVarietyMatrixForm}, just inserting
    $\alpha_i$ into the polynomials. So
    $\Fourier(\kappa^{-1}(J)) = \{0\}$, hence $\kappa^{-1}(J) = \{0\}$
    in $\algebra{A}$ and evidently in $\algebra{B}$ aswell, as
    $\algebra{B}$ is a subalgebra of $\algebra{A}$. Hence the
    Claim~(2) is proved.

    By Claim~(1) and Claim~(2) we have proved the proposition. \qed
\end{proof}
Hence in this case there always exists a transversal of $\algebra{B}$
in $\algebra{A}$, as one can choose each $t \in T$ such that
$t(a_i) = 0$ and $t(a_\ell) \not= 0$ for one $a_\ell \in \Variety(I)$.
Then each $t \algebra{B}$ has dimension $1$, and hence
$\dim \bigoplus_{t \in T} t \algebra{B} = \dim \bigoplus_{a \in
  \Variety(I)} \polynomials{n}(x) \big/ \ideal{x_1 - a_1,\dots, x_n -
  a_n}$. Thus they are isomorphic as vector spaces and by
\eqref{eq:CRTFiniteVariety} to $\algebra{A}$ aswell. Note that this
choice is a useless one for the development of fast algorithms, as we
have no intermediate steps and hence one does not obtain a recursive
structure which can be exploited for speeding up calculations.
Nonetheless this is a necessary remark, as now we can always assume a
transversal existent.

Choose a transversal $T$ of the subalgebra $\algebra{B}$ in
$\algebra{A}$. The next step is to show that the structure
$\algebra{B}$-modules $t \acts M$ for $\algebra{B}$-modules of the
form $M =  \polynomials{n}(y) \big/ J$ with zero-dimensional,
radical ideal $J$ and $t \in T$ is again a polynomial module. Hence
one gets a descending chain of submodules where one can easily the
describe the corresponding Fourier transforms.
\begin{proposition}
    \label{proposition:StrucutreOfTransversedModules}%
    Let $\algebra{A}$ be an algebra with subalgebra $\algebra{B}$ and
    let $T$ be a finite transversal of $\algebra{B}$ in $\algebra{A}$.
    Let $M =  \polynomials{n}(y) \big/ J$ be a
    $\algebra{B}$-module and
    $\bigoplus t \acts M = \polynomials{n} \big/ I$ the
    induced $\algebra{A}$-module.
    There exists a map $r \colon \Variety(I) \longrightarrow
    \Variety(J)$. 
    The action of the transversal elements leads to
    $\algebra{B}$-modules of the form
    \begin{equation}
        \label{eq:StructuredOfTransversedModules}
        t \acts M \cong  \mathbb{C}[y_1,\dots,y_n] \big/ J_{t},           
    \end{equation}
    where
    $J_{t} = \Ideal( \{ r(\alpha) \; | \; \alpha \in \Variety(I)
    \text{ and } t_p(\alpha) \not= 0 \}$.
\end{proposition}
\begin{proof}
    The existence of the map $r$ is clear, since $\algebra{B}$ is a
    subalgebra of $\algebra{A}$. Hence $T$ must contain $1$ and thus
    $M$ is a submodule of $\bigoplus t \acts M$. Therefore $r$ can be
    choosen as a projection of $\Variety(I)$ onto its subset
    $\Variety(J)$.

    It suffices to show that the
    $\algebra{B}$-modules on both sides of
    \ref{eq:StructuredOfTransversedModules} are of equal dimension.
    Then they are isomorphic as commutative algebras have the
    invariant basis property and the all appearing modules are free.

    The isomorphism from the Chinese remainder theorem for $\bigoplus
    t \acts M$ leads for the subset $t \acts M$ to 
    \begin{equation*}
        t_p p
        \mapsto (t_p(\alpha) p(r(\alpha)) )_{\alpha \in \Variety(I)},
    \end{equation*}
    for any $p \in M$. Denote by $[\alpha]$ the equivalence class of
    $\alpha \in \Variety(I)$ which map to the same
    $\beta \in \Variety(J)$. The dimension of $t \acts M$ is
    $|\Variety(J)|$ minus one for each $[\alpha]$ where
    $t_p(\alpha) = 0$. Restricting to $J_t$ hence does not change the
    dimension. The proposition is proven. \qed
\end{proof}
\begin{remark}
    \label{remark:MapBetweenVarieties}%
    Note that the map
    $r \colon \Variety(I) \longrightarrow \Variety(J)$ from
    Prop.~\ref{proposition:StrucutreOfTransversedModules} can
    explicitly determined if $M$ is a \emph{subalgebra} of $N$. Then
    the map is just the set of generators $r = (r_1, \dots,r_n)$ from
    Prop.~\ref{proposition:CharacterizationOfPolynomialSubalgebras}.
\end{remark}
We can not give a general statement about the computational cost of
these algorithms, as in general we do have only the trivial $O(n^2)$
estimate for the computational cost of the matrices $B$ and $M_i$. But
if we assume them to be of linear cost and if we can find a suitable
descending chain of submodules these algorithms are of cost
$O(k \log(k))$, where $k = |\Variety(I)|$. Then the following
proposition is a simple consequence of the
Akra-Bazzi-Theorem~\cite{Akra.Bazzi:1998a}, a refined version of the
Master Theorem for divide and conquer
recurrences~\cite{Bentley.Haken.Saxe:1980a}.
\begin{proposition}
    \label{proposition:ComputationalCostCooleyTukey}%
    Consider the decomposition of the Fourier transform from
    Theorem~\ref{theorem:CooleyTukeyTypeAlgorithms} and assume one has
    a desceding chain of submodules, where in each step we have a
    split in at least two submodules. If the basis change matrices $B$
    and the $M_i$ in each step are $O(k)$ then the decomposition is
    $O(k \cdot \log(k))$.
\end{proposition}
Finding a descending chain of submodules is no problem as one can
collect random points of the variety but this typically leads to
neither sparse $B$ nor sparse $M$. Hence the main difficulty for an
effective applications of the theorem is finding good examples.

For a fast recursive algorithm one needs a chain of descending
submodules. The decomposition property $p(x) = q(r(x))$ is very useful for the
development of fast algorithms as from the following proposition one
obtains a nice chain of subalgebras. The several variables analog of
the decomposition property reads
\begin{equation}
    \label{eq:DecompositionProperty}
    (p_1,\dots,p_n) = (q_1(r_1,\dots,r_n),\dots,q_n(r_1,\dots,r_n)).
\end{equation}
Since this notation is rather opulently, we write
$\ideal{p} = \ideal{p_1,\dots,p_n}$ and
$\ideal{q(r)} = \ideal{q_1(r_1,\dots,r_n),\dots,q_n(r_1,\dots,r_n)}$
if confusion with the one-variable case can be avoided by context. The
decomposition property yields the existence of sufficiently
well-behaved submodules.
\begin{proposition}
    \label{proposition:DecompositionOfIdeals}%
    Assume the zero-dimensional radical ideal
    $I = \ideal{p_1,\dots,p_n}$ satisfies
    \begin{equation*}
        \ideal{p} = \ideal{q(r)}.
    \end{equation*}
    Then $\ideal{r} \cong \polynomials{n}(y) \big/ \ideal{q}$.
\end{proposition} 
\begin{proof}
    The mapping $(r_1,\dots,r_n)$ maps $\Variety(I)$ to the variety of
    the $q_1,\dots,q_n$, i.e.\
    $(r_1,\dots,r_n) (\Variety(I)) = \Variety(\ideal{q})$, as
    $\ideal{p} = \ideal{q(r)}$. By
    Proposition~\ref{proposition:CharacterizationOfPolynomialSubalgebras}
    one has $\ideal{r} \cong \polynomials{n}(y) \big/ \ideal{q}$. Thus
    the proposition is proven. \qed
\end{proof}
In the univariate case one can always obtain a transversal of the
algebra $\ideal{r(x)} \cong \mathbb{C}[y] \big/ \ideal{q(y)}$ from a
basis of $\mathbb{C}[x] \big/ \ideal{r(x)}$. In the multivariate case
this is not always the case. The next proposition formalizes this in
terms of the appearing varieties. Sect.~\ref{sec:UndirectedLattices}
contains examples for both situations.
\begin{proposition}
    \label{proposition:ExistenceOfTransversalByDecomposition}%
    Consider
    $ \ideal{p_1,\dots,p_n} = \ideal{q_1(r_1,\dots,r_n), \dots,
      q_n(r_1,\dots,r_n)}$ with zero-dimensional variety. If
    $|\Variety(\ideal{p})| \not= |\Variety(\ideal{r})| \cdot
    |\Variety(\ideal{q})|$ then no basis of
    $\polynomials{n} \big/ \ideal{r}$ is a transversal of
    $\polynomials{n}(y) \big/ \ideal{q}$ in
    $\polynomials{n} \big/ \ideal{p}$. If
    $|\Variety(\ideal{p})| = |\Variety(\ideal{r})| \cdot
    |\Variety(\ideal{q})|$ then any basis of
    $\polynomials{n} \big/ \ideal{r}$ is a transversal of
    $\polynomials{n}(y) \big/ \ideal{q}$ in
    $\polynomials{n} \big/ \ideal{p}$.
\end{proposition}
\begin{proof}
    If
    $|\Variety(\ideal{p})| \not= |\Variety(\ideal{r})| \cdot
    |\Variety(\ideal{q})|$ the dimensions of
    $\polynomials{n}(y) \big/ \ideal{q}$ and
    $\polynomials{n} \big/ \ideal{r}$ do not multiply to the dimension
    of $\polynomials{n} \big/ \ideal{p}$ so a basis of
    $\polynomials{n} \big/ \ideal{r}$ can not be a transversal of
    $\polynomials{n}(y) \big/ \ideal{q}$.

    For the second part observe that if $\{Q_1,\dots,Q_{q_d}\}$ is a
    basis of $\polynomials{n}(y) \big/ \ideal{q}$ and $\{R_1, \dots,
    R_{r_d}\}$ is a basis of $\polynomials{n} \big/ \ideal{r}$ then
    \begin{equation*}
        \begin{bmatrix}
            R_1 Q_1(r_1,\dots,r_n)  & \dots & R_1
            Q_{q_d}(r_1,\dots,r_n)  \\ 
            \vdots & & \vdots \\
            R_{r_d} Q_1(r_1,\dots,r_n)  & \dots & R_{r_d}
            Q_{q_d}(r_1,\dots,r_n)  
        \end{bmatrix}
    \end{equation*}
    is a basis of $\polynomials{n} \big/ \ideal{p}$ if
    $|\Variety(\ideal{p})| = |\Variety(\ideal{r})| \cdot
    |\Variety(\ideal{q})|$. Hence $\{R_1, \dots, R_{r_d}\}$ is an
    induction of $\polynomials{n}(y) \big/ \ideal{q}$ in
    $\polynomials{n} \big/ \ideal{p}$. \qed
\end{proof}
Even though this renders some of the ideals obeying the decomposition
property~\eqref{eq:DecompositionProperty} useless for their
application with the decomposition
Theorem~\ref{theorem:CooleyTukeyTypeAlgorithms} for Fourier
transforms, the decomposition property is a useful one since there is
another decomposition theorem for the Fourier transform. This version
is the correct version of \cite[Thm.~3]{Pueschel.Roetteler:2008} if
one does not assume that the sizes of the varieties of the decomposed
ideals multiply to the size of the original variety.
\begin{theorem}[FFT algorithms, top-down approach]
    \label{theorem:CooleyTukeyFFTUsingDecomposition}%
    Let $\algebra{A} = \polynomials{n} \big/ \ideal{p}$ such that
    $\ideal{p} = \ideal{q(r)}$ and consider the signal model with
    regular module $N = \algebra{A}$. Let $k = |\Variety(\ideal{q})|$.
    Denote by $M_{\alpha} = \polynomials{n} \big/ \ideal{r - \alpha}$
    for $\alpha \in \Variety(\ideal{q})$. Denote for $i=1,\dots,k$ by
    $d_i = \dim M_{\alpha_i}$, ordered with respect to size. The
    Fourier transform of $N$ with respect to a basis $b$ can then be
    decomposed as
    \begin{equation}
        \label{eq:CooleyTukeyFFTUsingDecomposition}
        P_{b,N} = P \cdot \left( \bigoplus_{i} P_{M_{\alpha_i}}
        \right) 
        \cdot T
        \cdot B,
    \end{equation}
    where $P$ is permutation matrix, $P_{M_{\alpha_i}}$ are the
    Fourier transforms of each $M_{\alpha_i}$, $B$ is a basis change
    between bases of $N$. Denote by $(c_{i,j})$ the entries of the
    Fourier transform of $\polynomials{n} \big/ \ideal{q}$. The matrix
    $T$ is a block matrix of the form
    \begin{equation}
        \label{eq:PseudotensorProductOneWithFourier}
        \begin{bmatrix}
            c_{i,j} \mathbbm{1}_{\min(d_i,d_j)} & 0_{d_i,d_j} \\
            0_{d_j,d_i}^\top
        \end{bmatrix}_{i,j = 1,\dots,k},
    \end{equation}
    where $0_{d_i,d_j}$ is the (possible empty) $d_i \times \max(0,d_j -
    d_i)$ zero matrix.  
\end{theorem}
\begin{proof}
    By the decomposition property there exists a basis of $N$ of the
    form
    \begin{equation*}
        \begin{bmatrix}
            t_{1,1} u_1(r(x)), \dots, t_{1,d_1} u_1(r(x)) \\
            t_{2,1} u_2(r(x)), \dots, t_{2,d_2} u_2(r(x)) \\
            \vdots \\
            t_{k,1} u_k(r(x)), \dots, t_{k,d_k} u_k(r(x)) 
        \end{bmatrix}.
    \end{equation*}
    The basis change is from $b$ to this basis. The isomorphism
    $\polynomials{n} \big/ \ideal{q(r)} \longrightarrow \bigoplus_i
    M_{\alpha_i}$ is, using that basis, realized by $T$. By the
    decomposition property the zeros of the $M_{\alpha_i}$ are the
    zeros of $N$, except possible in a different ordering. The theorem
    follows. \qed
\end{proof}
If each $M_\alpha$ is of equal dimension the matrix $T$ is just the
tensor product of the Fourier transform of
$\polynomials{n} \big/ \ideal{q}$ with $\mathbbm{1}_k$.

For the same reasons as in
Prop.~\ref{proposition:ComputationalCostCooleyTukey} one gets again a
fast algorithm if the basis change is sparse and one has a descending
chains of submodules with the decomposition property.

We now give an example, the FFT on a directed hexagonal lattice, which
shows how one can derive FFTs on various lattices from the literature.
The derivation of FFTs on regular directed lattices was first obtained
in~\cite{Mersereau.Speake:1981a}. See~\cite{Zheng.Gu:2014} for more
concrete examples using the classical derivation. The example
illustrates a reverse engineering approach to obtain these algorithms
by algebraic signal processing theory, aswell.
\begin{example}
    \label{example:MersereauHexagonalFFT}%
    We reverse engineer the FFT of a directed hexagonal lattice from
    Mersereau\cite{Mersereau:1979a} by algebraic signal processing
    theory. Assume $N = 2^k$ for some $k > 1$. Recall from
    \cite{Mersereau:1979a} that the discrete Fourier transform for a
    signal $s_{n_1,n_2}$ sampled on a hexagonal lattice is given as
    \begin{equation}
        \label{eq:HexagonalDFT}
        \Fourier(s)_{k_1,k_2} = \sum_{n_1 = 0}^{3 N - 1} \sum_{n_2 = 0}^{N-1}
        s_{n_1,n_2} \exp(- \tfrac{-\pi \I}{3 N} ( (2 n_1 - n_2) (2 k_1-k_2) +
        6 n_2 k_2)). 
    \end{equation}
    From this formula and the definition of Fourier transform
    corresponding to a zero-dimensional
    varieties~\eqref{eq:FourierTrafoFiniteVariety} it is evident that
    the variety is given by the points
    \begin{equation}
        \label{eq:HexagonalFFTVariety}
        \{ (\exp(\tfrac{- \pi \I (2 k_1-k_2)}{3 N}), \exp(\tfrac{2 \pi \I
          k_2}{N})) \; | \; k_1 = 0,\dots, 3 N - 1 \quad k_2 = 0,\dots, N
        - 1\}.  
    \end{equation}
    The basis is determined by \eqref{eq:HexagonalDFT}, aswell, and
    consist of elements $x^{2 n_1 - n_2} y^{n_2}$ for
    $n_1 = 0,\dots,3 N -1$ and $n_2 = 0,\dots, N - 1$.

    The vector space underlying the module is hence given by
    \begin{equation}
        \label{eq:ModuleDirectedHexagonalFFT}
        M = \mathbb{C}[x^2, x y] \big/ \ideal{y^{N} - 1, x^{3 N} -
          y^{N/2}}.
    \end{equation}
    Now we have to expose for which algebra we can find a module
    structure, such that we get a hexagonal model and a FFT-like
    algorithm. Unlike one might speculate at first, one realizes the
    module structure of $M$ not as a module over a polynomial algebra
    in two variables but in three. For this consider the algebra
    $\algebra{A} = \mathbb{C}[X_1,X_2,X_3] \big/ \ideal{X_1^{3 N} - 1,
      X_2^{N/2} - 1, X_3^{N/2} - 1}$, with actions on $M$ given by
    \begin{equation}
        \label{eq:AlgebraActionDirectedHexagonalLattice}%
        \begin{split}
            X_1 \acts p(x,y) &= x^2 \cdot p(x,y), \\
            X_2 \acts p(x,y) &= x y \cdot p(x,y), \\
            X_3 \acts p(x,y) &= x^{-1} y \cdot p(x,y).            
        \end{split}
    \end{equation}
    The resulting visualization graph of the signal model is shown in
    Fig.~\ref{fig:DirectedHexagonalModel}.    

    The signal module can be decomposed in submodules. The choice of
    lattice cosets in \cite{Mersereau:1979a} corresponds to the choice
    of the submodule
    $S = \mathbb{C}[r^2,r s] \big/ \ideal{ s^{N/2} - 1, r^{3 N/2} -
      s^{N/4} }$ with $r = x^2$ and $s = y^2$. We need to find a
    subalgebra and transversal of the underlying algebra, which
    results in the induced module of $S$ being $M$. Consider the
    subalgebra
    $\algebra{B} = \mathbb{C}[Y_1,Y_2,Y_3] \big/ \ideal{Y_1^{3 N/2} -
      1, Y_2^{N/4} - 1, Y_3^{N/4} - 1}$. A transversal of
    $\algebra{B}$ in $\algebra{A}$ is $\{1,X_1,X_2,X_3\}$. The action
    of the transversal elements on $S$ is realized by multiplication
    with the polynomials $\{1,x^2, x y, x^{-1}y\}$. Then one obtains
    \begin{equation}
        \label{eq:HexagonalFFTInducedModule}
        M = S + x^2 S + x y S + x^{-1} y S.
    \end{equation}
    The sublattice corresponding to the transversal element $1$ is
    depicted in Fig.~\ref{fig:DirectedHexagonalLatticeDecomposed}.
    \begin{sidewaysfigure}
        \centering     \rule{0pt}{0.6\textheight}
        \includegraphics[width=\textwidth]{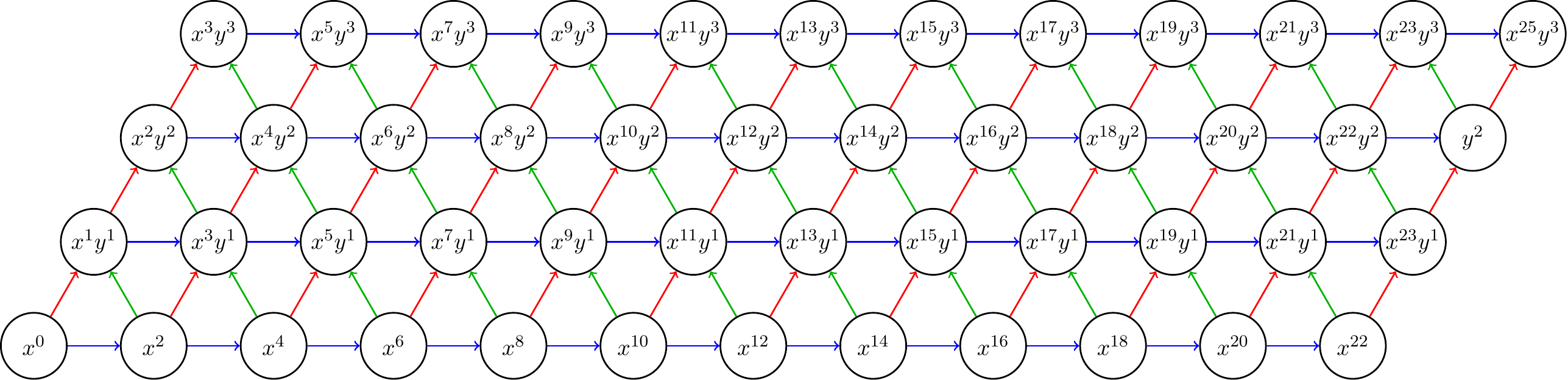}
        \caption{Signal model of a directed hexagonal lattice for
          $N = 4$. The boundary conditions are omitted. Shifts of $X$
          are blue, shifts of $Y$ are red, and shifts of $Z$ are green
          colored.}
        \label{fig:DirectedHexagonalModel}
        \vspace{0.05\textwidth}
        \includegraphics[width=\textwidth]{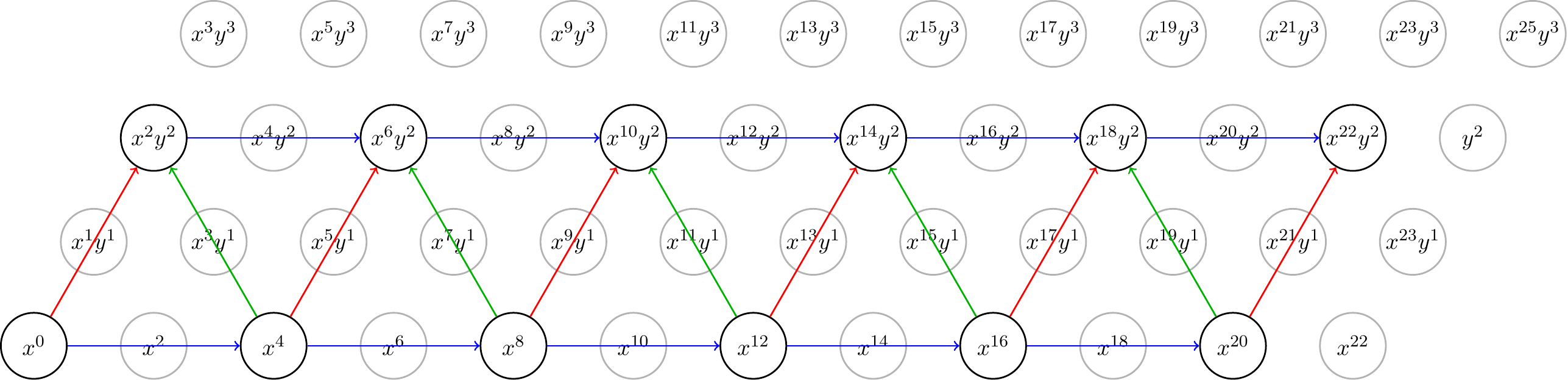} 
        \caption{Sublattice of the hexagonal lattice
          corresponding to the transversal element $1$.}
        \label{fig:DirectedHexagonalLatticeDecomposed}
    \end{sidewaysfigure}
    From the structure of the submodule and the transversal it is
    obvious that the basis change to the induced module is a
    permutation matrix, hence is sparse.

    None of the elements of
    $\Variety(\ideal{y^{N} - 1, x^{3 N} - y^{N/2}})$ gets mapped to
    zero by an element of the transversal. The preimage of each point
    of $V(\ideal{y^{N/2} - 1, x^{3 N/2} - y^{N/4}})$ consists at most
    of four points of $V(\ideal{y^{N} - 1, x^{3 N} - y^{N/2}})$. Hence
    each row of $M$ has at most $4$ non-zero entries, thus $M$ has
    $O(n)$ entries and is sparse. Thus by
    Prop.~\ref{proposition:ComputationalCostCooleyTukey} we have a
    fast algorithm.
\end{example}
\begin{remark}
    \label{remark:HexagonalVSQuincunxLattice}%
    In \cite{Pueschel.Roetteler:2005a} a signal model for the directed
    quincunx lattice was introduced. This signal model used a basis
    similar to the one we used in
    Example~\ref{example:MersereauHexagonalFFT}. These examples show
    that the algebra action on the module is indeed crucial for the
    signal model.
\end{remark}


\section{Orthogonal polynomials and orthogonal transforms}
\label{sec:ChristoffelDarboux}%

Orthogonal polynomials are at the heart of numerical mathematics. In
this section we recall some properties of them, focusing on the
multivariate case. Especially interesting for their usage in algebraic
signal processing are the three-term recurrence relations and the
Christoffel-Darboux formula. Since the theory of multivariate
orthogonal polynomials relies on a formulation not in special bases
but spaces of polynomials of the same degree, we introduce a notion of
equivalence of signal models to capture that ambiguity. Then the
multivariate Christoffel-Darboux formula is used to derive a
multivariate version of the Gauß-Jacobi procedure for finding
orthogonal Fourier transforms.

We then recall the notion of Gaussian cubature. Whilst in the
univariate case Gaussian cubature formulae always exist, this is not
the case in the multivariate setting. Indeed there are few domains
known for which such a formula can be stated. We show that the
existence of a orthogonal Fourier transform for a signal model with
orthogonal polynomials as basis is implied by the existence of a
Gaussian cubature.

In this section we denote by $x = (x_1,\dots,x_d)$. Denote by
$\polynomials[n]{d}$ the space of polynomials in $d$ variables of
degree at most $n$. Let $\SP{\argument, \argument}$ be an inner
product on $\polynomials{d}$. The space of orthogonal polynomials of
degree $n$ is denoted by
\begin{equation}
    \label{eq:OrthogonalPolynomials}
    \orthopolys[n]{d} = \{ p \in \polynomials[n]{d} \; | \; \SP{p,q} = 0
    \text{ for all } q \in \polynomials[n-1]{d} \}.
\end{equation}
If the inner product is given by
$\SP{f,g} = \int_{\mathbb{R}^d} f(x) g(x) \D \mu(x)$ such that the
measure $\mu$ has support with non-empty interior, then
\begin{equation}
    \label{eq:DimensionOrthogonalPolys}
    \dim \orthopolys[n]{d} = \binom{n+d-1}{n} = r_n^d.
\end{equation}
Let $\mathbb{P}_n = (P_{\alpha}^n)_{|\alpha| = n}$, with
$\alpha \in \mathbb{N}_0^d$ a multi-index, denote a basis of
$\orthopolys[n]{d}$ and let $\mathbb{P}_{-1}(x) = 0$. There exist
unique matrices $A_{n,i}$ of size $r_n^d \times r_{n+1}^d$, $B_{n,i}$
of size $r_d^d \times r_n^d$, and $C_{n,i}$ of size
$r_n^d \times r_{n-1}^d$ such that one has the three-term recurrence
relation
\begin{equation}
    \label{eq:ThreeTermRecurrenceRelation}
    x_i \mathbb{P}_n(x)
    = A_{n,i} \mathbb{P}_{n+1}(x) + B_{n,i} \mathbb{P}_{n}(x) +
    C_{n,i} \mathbb{P}_{n-1}.
\end{equation}
The matrices $A_{n,i}$ and $C_{n,i}$ are of full rank. If the
$\mathbb{P}_n$ are even orthonormal one has $C_{n,i} =
A_{n-1,i}^\top$.

From this three-term recurrence relation Xu deduced a multivariate 
Christoffel-Darboux formula~\cite{Xu:1993a}, which reads
\begin{equation}
    \label{eq:ChristoffelDarboux}
    \begin{split}
    &\sum_{k=0}^n \mathbb{P}_k^{\top}(x) H_k^{-1} \mathbb{P}_k(y) \\
    &= 
    \begin{cases}
        \frac{(A_{n,i} \mathbb{P}_{n+1}(x))^\top H_n^{-1}
          \mathbb{P}_n(y) - \mathbb{P}_n^\top(x) H_n^{-1} A_{n,i}
          \mathbb{P}_{n+1}(y)}{x_i - y_i}
        &\text{if } x \not= y \\
        \mathbb{P}_n^\top(x) H_n^{-1} A_{n,i}
        \frac{\partial}{\partial x_i} \mathbb{P}_{n+1}(x) - (A_{n,i}
        \mathbb{P}_{n+1}(x))^\top H_n^{-1} \frac{\partial}{\partial
          x_i} \mathbb{P}_n(x)
        &\text{if } x = y,
    \end{cases}        
    \end{split}
\end{equation}
where the $H_k$ are invertible, symmetric matrices such that
$B_{k,i} H_k$ is symmetric and one has
$A_{k,i} H_{k+1} = H_k C_{k+1,i}^\top$. The matrices $H_k$ are given
as $H_k = \mathcal{L}(\mathbb{P}_k \mathbb{P}_k^\top)$, with
$\mathcal{L}(f) = \int_{\mathbb{R}^d} f(x) \D \mu(x)$. Note that the
value of the sum
$\sum_{k=0}^n \mathbb{P}_k^{\top}(x) H_k^{-1} \mathbb{P}_k(y)$ is
independent on the actual choice of the bases $\mathbb{P}_k$ in the
$\orthopolys[k]{d}$. This follows from the equality
$\mathbb{P}^\top \mathcal{L}(\mathbb{P}_k \mathbb{P}_k^\top)
\mathbb{P} = \mathbb{Q}^\top \mathcal{L}(\mathbb{Q}_k
\mathbb{Q}_k^\top) \mathbb{Q}$ for any choice of bases
$\mathbb{P},\mathbb{Q}$ in $\orthopolys[k]{d}$, cf. \cite{Xu:1993a}.
Note that even though it appears from the right-hand side of the
formula that it depends on the choice of index $i$, the left-hand side
shows that its value actually is independent of $i$.

Another nice property of orthogonal polynomials is that their common
zeros are particular well-behaved. Recall that a common zero of
$\mathbb{P}_n$ is a zero of all the $P_{\alpha}^n$ in $\mathbb{P}_n$.
All common zeros of $\mathbb{P}_n$ are real, distinct and simple, i.e.
at least one $\tfrac{\partial}{\partial x_i} \mathbb{P}_n$ does not
vanish and the set $\mathbb{P}_n$ has at most
$\dim \polynomials[n-1]{d}$ common zeros, cf.~\cite{Xu:2017a}.

We now adopt the point of view, that orthogonality does not hold in
terms of particular bases of $\orthopolys[n]{d}$ but in terms of the
subspaces $\orthopolys[n]{d}$, to algebraic signal models. 
\begin{definition}
    \label{definition:InsignificantDifferentSignalModels}%
    Two signal models $(\algebra{A},M,\Phi_1)$ and
    $(\algebra{A},M,\Phi_2)$, with bases of the modules given by sets
    of orthogonal polynomials $\mathbb{P}_1$ and $\mathbb{P}_2$, are
    called \emph{insignificantly different} if $\mathbb{P}_1$ and
    $\mathbb{P}_2$ are orthogonal with respect to the same positive
    definite linear functional.
\end{definition}
From a signal processing perspective it is interesting when one can
obtain an orthogonal transform. In the univariate case the Gauß-Jacobi
procedure~\cite{Yemini.Pearl:1979a} shows that one can always obtain
an orthogonal transform if the basis of the signal module consists of
orthogonal polynomials. In the multivariate case one has to assume an
additional condition on the number of common zeros of the orthogonal
polynomials.
\begin{theorem}
    \label{theorem:EquivalenceGaussianCubatureAndOrthogonalTransforms}%
    Consider a signal model with underlying variety
    $V = \Variety(\ideal{\mathbb{P}_n})$ with
    $|V| = \dim \polynomials[n-1]{d}$ and let the basis of the
    module be given by $(\mathbb{P}_i)_{i=0}^{n-1}$. Then there exists
    an insignificantly different signal model with orthogonal Fourier
    transform.
\end{theorem}
\begin{proof}
    If the variety of the signal model consists of the nodes of such a
    Gaussian cubature formula, the underlying variety $\Variety$
    consists of common zeros of all $\mathbb{P}_n$. We can assume that
    all the $(\mathbb{P}_i)_{i=0}^{n-1}$ are orthonormal since this
    choice only leads to an insignificantly different signal model.
    Then the product of the Fourier transform matrices for this signal
    model $\Fourier^\top \cdot \Fourier$ has entries of the form
    $\sum_{k=0}^{n-1} \mathbb{P}_k(\alpha)^\top \mathbb{P}_k(\beta)$
    for $\alpha, \beta \in \Variety$. Now by the Christoffel-Darboux
    formula~\eqref{eq:ChristoffelDarboux} the entries not on the
    diagonal are zero since $\mathbb{P}_n(\alpha) = 0$ for each
    $\alpha \in \Variety$. On the other hand the diagonal entries have
    the form
    \begin{equation*}
        \mathbb{P}_{n-1}^\top(\alpha) A_{n-1,i}
        \tfrac{\partial}{\partial x_i} \mathbb{P}_n(\alpha).
    \end{equation*}
    Since the common zeros of $\mathbb{P}_n$ are simple, i.e. at least
    one partial derivative of $\mathbb{P}_n$ is not zero, cf.
    \cite[Thm.~2.13]{Xu:2017a}, we can invert the diagonal entries,
    which do not depend on $i$. If we now choose in the
    one-dimensional, irreducible component belonging to $\alpha$ the
    basis
    $\Big(\sqrt{\mathbb{P}_{n-1}^\top(\alpha) A_{n-1,i}
      \tfrac{\partial}{\partial x_i} \mathbb{P}_n(\alpha)}\Big)$ we obtain
    an orthogonal Fourier transform $\Fourier^{\mathsf{orth}}$. This
    can be seen as follows. Consider the diagonal matrix
    \begin{equation*}
       \sqrt{D} = \mathsf{diag}\left(1 \Big/ \sqrt{\mathbb{P}_{n-1}^\top(\alpha)
      A_{n-1,i} \tfrac{\partial}{\partial x_i} \mathbb{P}_n(\alpha)} \;
    \bigg| \; \alpha \in \Variety\right). 
    \end{equation*}
    Then $\Fourier^{\mathsf{orth}} = \sqrt{D} \Fourier$, hence we
    obtain from the above discussion
    \begin{equation*}
        \Fourier^{\mathsf{orth}, \top} \cdot \Fourier^{\mathsf{orth}}
        = \Fourier^\top \sqrt{D} \sqrt{D} \Fourier
        = D \Fourier^\top \Fourier
        = \mathbbm{1},
    \end{equation*}
    since diagonal matrices commute with all matrices. The theorem is
    proven. \qed
\end{proof}
Now the condition that one has
$|\Variety(\ideal{\mathbb{P}_n})| = \dim \polynomials[n-1]{d}$ is very
restrictive. Indeed it is the same condition as for the existence of a
Gaussian cubature formula and there are few multi-dimensional regions
known for which Gaussian cubature formulas exist.

Recall that a cubature formula for the measure $\mu$ is a finite sum
that approximates integrals $\int_{\mathbb{R}^d} \argument \D \mu$. If
one has
\begin{equation}
    \label{eq:CubatureFormula}
    \int_{\mathbb{R}^d} f(x) \D \mu
    = \sum_{k=1}^N w_k f(x_k),
\end{equation}
with weights $w_k \in \mathbb{R}$ and nodes $x_k \in \mathbb{R}^d$,
for all $f \in \polynomials[2n-1]{d}$ and this does not hold for at
least one element of $\polynomials[2n]{d}$, the cubature is said to be
of degree $2n-1$. For the number of nodes $N$ one has
\begin{equation}
    \label{eq:NumberOfCubatureNodesLowerBound}
    N \geq \dim \polynomials[n-1]{d}
\end{equation}
and if the bound is reached the cubature formula is called Gaussian. A
Gaussian cubature formula exists if and only if $\mathbb{P}_n$ has
$\dim \polynomials[n-1]{d}$ common zeros~\cite{Mysovskikh:1976a}. The
nodes of the cubature formula are then precisely the common zeros of
$\mathbb{P}_n$. Hence if there exists a Gaussian cubature formula one
can ensure the existence of an orthogonal transform for a
corresponding signal model. Now the existence of Gaussian cubature
formulas is rare, the first class of examples in any dimension has
been described in \cite{Berens.Schmid.Xu:1995a} and other examples
have been discussed in
\cite{Li.Sun.Xu:2008,Moody.Patera:2011,Hrivnak.Motlochova.Patera:2016}.
Thus the applicabilty of the multivariate Gauss-Jacobi procedure is
restricted to certain special cases. One of these special cases will
be investigated in the next section.

\section{FFT for weight lattices}
\label{sec:UndirectedLattices}%

While directed signals are of interest in the analysis of
time-dependent data like time-series, undirected signals are
considered in the analysis of space-dependent data like images.

In the one-dimensional case the undirected counterparts to the
directed discrete Fourier transform are the discrete sine and cosine
transforms. In \cite{Pueschel.Moura:2008c} the signal models and fast
transforms for all 16 discrete sine and cosine transforms were deduced
in algebraic signal processing. In this section undirected signal
models and their fast transforms for a special class of lattices, the
weight lattices of semi-simple Lie groups, are derived. The approach
mimics the ansatz of \cite{Pueschel.Moura:2008c} for the DCT-3. This
ansatz relies on Chebyshev polynomials of the first kind. This family
of polynomials is one of the only two in one variable obeying the
decomposition property.

We start by recalling the ansatz for DCT-3. Consider the Chebyshev
polynomials of the first kind $T_n(\cos \theta) = \cos n \theta$. They
obey the shift property
\begin{equation}
    \label{eq:ShiftPropertyOneDimChebyshev}
    x T_n(x) = \tfrac{1}{2} ( T_{n-1}(x) + T_{n+1}(x)).
\end{equation}
Consider the signal model with filter algebra
$\algebra{A} = \mathbb{C}[x] \big/ \ideal{T_n}$,the regular module
$M = \algebra{A}$ as signals and
$\Phi \colon s \mapsto \sum_i s_i T_i$ determining the Chebyshev
polynomials as basis. By \eqref{eq:ShiftPropertyOneDimChebyshev} the
visualization of the signal model is an undirected lattice as
illustrated in Fig.~\ref{fig:DCTModel}.
\begin{figure}
    \centering
    \includegraphics[width=0.8\textwidth]{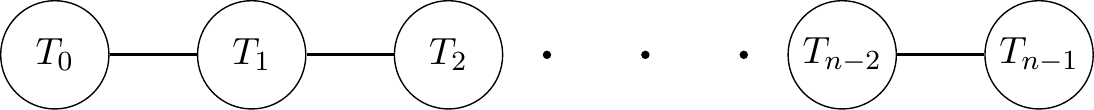}
    \caption{The visualization of the discrete cosine transform signal
      model.}
    \label{fig:DCTModel}
\end{figure}
The choice of the basis $\{T_0,T_1,\ldots, T_n\}$ leads to the
discrete cosine transform of type 3. The other types of discrete
cosine and sine transforms can be obtained by a combination of
different choices of kinds of Chebyshev polynomials and roots of them,
cf.~\cite{Pueschel.Moura:2003a}.

One particular nice property of the Chebyshev polynomials is the
decomposition property
\begin{equation}
    \label{eq:DecompositionPropertyChebyshevOneDim}
    T_{k \cdot n}(x) = T_n(T_k(x)).
\end{equation}
Up to similarity the Chebyshev polynomials and the monomials $x^n$ are
the only polynomials in one variable subject to the decomposition
property.

By the above considerations it is natural to search for several
variable analogues of the Chebyshev polynomials. Fortunately there is
a rather mature theory of multivariate Chebyshev polynomials
available~\cite{Hoffman.Withers:1988a} which has an intimate
connection to Lie theory. This generalization is based on the
stretching and folding property, a geometric interpretation of the
decomposition, of the one-dimensional Chebyshev polynomials, i.e. the
map
\begin{equation}
    \label{eq:ChebyshevFoldingProperty}
    \cos^{-1} \circ T_n \circ \cos
\end{equation}
stretches the interval $[0,1]$ $n$-times and folds it back at the
integers. In \cite{Hoffman.Withers:1988a} it was shown that the
foldable figures in higher dimensions are in one-to-one correspondence
to the Weyl groups of root systems.

For the correct generalisation of the appearing components we need to
recall some definitions and tools from Lie theory. This will include
an explanation of the domain of $\theta$ and the index set, aswell.

The first definition we need is that of a root system and its dual,
the coroot system. Root systems were introduced by Killing for the
classification of the complex, simple Lie
algebras~\cite{Killing:1888a}.
\begin{definition}
    \label{def:RootSystem}%
    A crystallographic root system in a finite-dimensional euclidean
    space $(\mathbb{R}^d, \SP{\argument, \argument})$ is a finite set
    $\roots$ of non-zero vectors, the so-called roots, which span
    $\mathbb{R}^d$ subject to the conditions
    \begin{enumerate}[i.)]
        \item  $r \cdot a \in \roots$ then $r = \pm 1$ for all $\alpha
        \in \roots$,
        \item closedness under reflections through the hyperplanes
        perpendicular to the roots, i.e.
        \begin{equation}
            \label{eq:ClosednessReflectionHyperplaneRoots}
            \sigma_\alpha(\beta)
            = \beta - 2 \frac{\SP{\alpha,\beta}}{\SP{\alpha,\alpha}} \alpha
            \in \roots
        \end{equation}
        for all $\alpha,\beta \in \roots$,
        \item for any
        $\alpha,\beta \in \roots$ we have
        $2 \tfrac{\SP{\alpha,\beta}}{\SP{\alpha,\alpha}} \in
        \mathbb{Z}$. 
    \end{enumerate}
    The set of integer linear combinations of the roots is termed root
    lattice
    \begin{equation}
        \label{eq:RootLattice}
        Q = \spann_{\mathbb{Z}}\{\alpha \in \roots\} \subseteq
        \mathbb{R}^d.  
    \end{equation}
    The coroot of a root
    $\alpha \in \roots$ is
    \begin{equation}
        \label{eq:Coroot}
        \alpha^\vee = \frac{2}{\SP{\alpha,\alpha}} \alpha.
    \end{equation}
    The coroots form a root system which is denoted by $\roots^\vee$.
    The coroot lattice $Q^\vee$ is the $\mathbb{Z}$-span of the
    coroots
\end{definition}
There are at most two different root lengths for an irreducible root
system, i.e.\ one which is not a combination of root systems with
mutually orthogonal spaces. The irreducible root systems can be
classified using Coxeter-Dynkin diagrams. There are four infinite
series $A_n, B_n, C_n, D_n$ and five exceptional root systems
$E_6, E_7, E_8, F_4, G_2$, cf. Fig.~\ref{fig:CoxeterDynkin}.
\begin{figure}
    \centering
    \includegraphics[width=\textwidth]{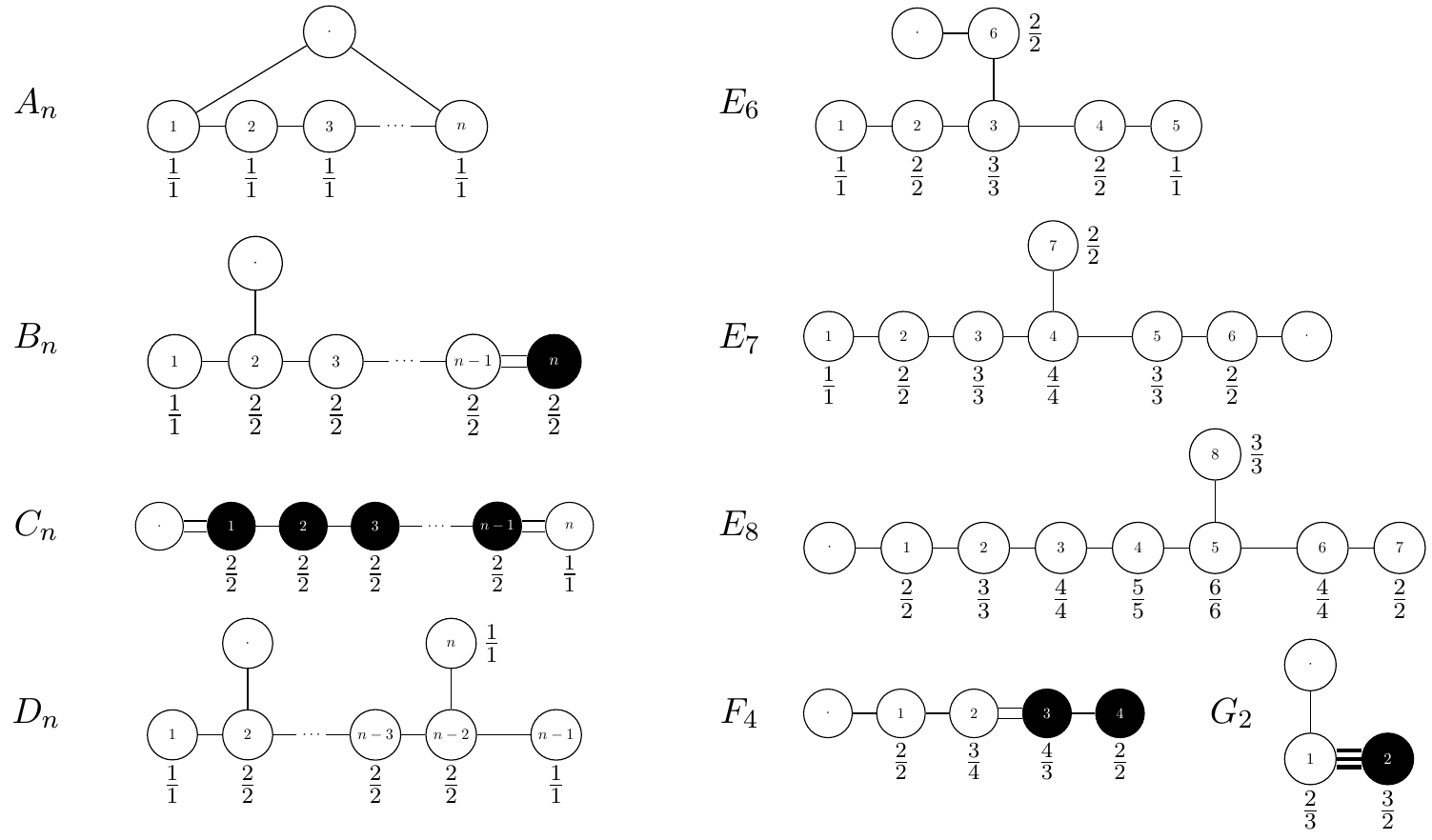}
    \caption{Affine Coxeter-Dynkin diagrams for the reduced,
      crystallographic root systems. The dotted node corresponds to
      the lowest root $-\alpha_0$, the numbered nodes to the simple
      roots $\alpha_i$. Open circles are long roots, while filled
      nodes indicate short roots. The marks and comarks are shown
      below the nodes as $\tfrac{m_i}{m_i^\vee}$. The angle between
      two roots depends on the multiplicity $k$ of the edge between
      them and is given as $4 \cos^2 \theta = k$ and
      $\cos \theta \leq 0$, i.e.
      $\tfrac{\pi}{2}, \tfrac{2 \pi}{3}, \tfrac{3 \pi}{4}, \tfrac{5
        \pi}{6}$ with length ratio being arbitrary, $1, \sqrt{2}, \sqrt{3}$
      for $k=0,1,2,3$, respectively. }
    \label{fig:CoxeterDynkin}
\end{figure}
One can choose a basis
$\simpleroots = \{\alpha_1,\dots,\alpha_d\} \subseteq \roots$ of the
root system such that one has $\alpha = \sum_{i=1}^d c_j \alpha_j$
with all $c_j \in \mathbb{Z}$ of the same sign. The $\alpha_i$ are
called simple roots. The simple roots divide the root system into
positive roots $\roots^+$ and negative roots $\roots^-$. The simple
roots introduces a partial order on the roots, aswell. The partial
order is defined by $\lambda \succeq \mu$ if the expansion of
$\lambda - \mu$ in simple roots has non-negative coefficients only.
Then $\lambda$ is called higher than $\mu$. The highest root is
\begin{equation}
    \label{eq:HighestRoot}
    \alpha_0 = m_1 \alpha_1 + \dots + m_d \alpha_d,
\end{equation}
with positive integers $m_i$. The $m_i$ are called the marks of the
root system. The marks of the coroot system are called the comarks of
the initial root system and denoted by $m_1^\vee, \dots, m_d^\vee$.

The Weyl group of a root system $\roots$ is the group generated by
the reflections
\begin{equation}
    \label{eq:WeylGroup}
    W = \ideal{ \sigma_\alpha \; | \; \alpha \in \roots}.
\end{equation}
The $\mathbb{Z}$-dual of $Q$ is the coweight lattice $P^\vee$, while
the $\mathbb{Z}$-dual of $Q^\vee$ is the weight lattice $P$. The
generators of $P$ are the fundamental weights $\omega_j$ and the
generators of $P^\vee$ are the fundamental coweights $\omega_j^\vee$.

The coroot lattice acts on $\mathbb{R}^d$ by translation and the
affine Weyl group is the semi-direct product
\begin{equation}
    \label{eq:AffineWeylGroup}
    W_{\text{aff}} = W \ltimes Q^\vee. 
\end{equation}
The simplex $F = \mathbb{R}^d \big/ Q^\vee$ tiles $\mathbb{R}^d$ under
the action of the affine Weyl group and is called the fundamental Weyl
chamber. One can use the fundamental coweights to describe the
fundamental Weyl chamber as convex hull
\begin{equation}
    \label{eq:FundamentalRegion}
    F = \mathsf{conv}\left\{0,\frac{\omega_1^\vee}{m_1}, \dots,
        \frac{\omega_d^\vee}{m_d} \right\}.    
\end{equation}
This fundamental region replaces the interval $[0,1]$ as stretching
and folding region for multivariate Chebyshev polynomials. In
Fig.~\ref{fig:RootsAndWeightsFundamentalDomain} the root systems of
type $A_2$ and $C_2$ are shown together with the simple scaled
coweights and the fundamental domains.
\begin{figure}
    \centering
    \includegraphics[width=0.65\textwidth]{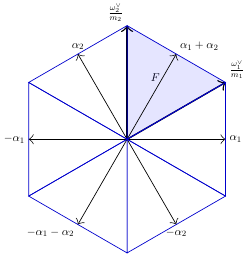}

    \vspace{0.05\textwidth}

    \includegraphics[width=1\textwidth]{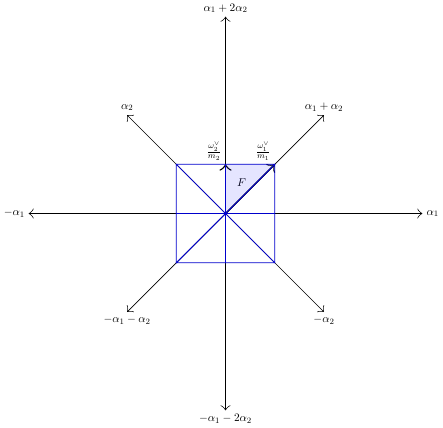}
    \caption{The root systems of type $A_2$ (upper) and $C_2$ (lower)
      together with the fundamental region $F$ (shaded region) and the
      image of $F$ under the action of the Weyl group.}
    \label{fig:RootsAndWeightsFundamentalDomain}
\end{figure}

The dual pairing
$\dualpair{\argument, \argument} \colon P \times \mathbb{R}^d \big/
Q^\vee \longrightarrow \mathbb{C}$ is given by
\begin{equation}
    \label{eq:DualPairingWeightsAndFundamentalDomain}
    \dualpair{\lambda, \theta} = \exp(2 \pi \I \SP{\lambda, \theta}). 
\end{equation}
The Weyl group, which is isomorphic to a group of integer matrices,
acts on $P$ and $\mathbb{R}^d \big/ Q^\vee$. Symmetrization of the
dual pairing with respect to the corresponding Weyl group now leads to
the definition of multivariate Chebyshev polynomials.
\begin{definition}
    \label{def:MultivariateChebyshevPolynomials}%
    Let $W$ be a Weyl group of a root system $\roots$ with weight
    lattice $P$ and coroot lattice $Q^\vee$. The multivariate
    Chebyshev polynomials of the first kind of weight $\lambda \in P$
    is
    \begin{equation}
        \label{eq:MultivariateChebyshevPolynomials}
        T_\lambda(x_1,\dots,x_d) = \frac{1}{|W|} \sum_{w \in W}
        \dualpair{\lambda, w \theta},
    \end{equation}
    for $\theta \in F$. The multivariate
    Chebyshev polynomials are polynomials in the variables
    \begin{equation}
        \label{eq:VariablesForMultivariateChebyshevPolynomials}
        x_k = \dualpair{\omega_k, \theta},
    \end{equation}
    with $\theta \in F$. 
\end{definition}
Indeed, for the Weyl group $W(A_1)$ of type $A_1$ one gets back the
original definition of Chebyshev polynomials, as the root system is
then $\roots_{A_1} = \{1,-1\}$ and is equal to the coroot lattice, the
only simple root is $\{1\}$, the Weyl group is $W(A_1) = \{1,-1\}$,
the root and coroot lattice is $\mathbb{Z}$, the weight lattice is
$P = \mathbb{Z}$. This leads to $F = [0,1]$
and
$T_n(x) = \dualpair{n,\theta} = \tfrac{1}{2} (\exp(2 \pi \I n \theta)
+ \exp(-2 \pi \I n \theta) = \cos(n \theta)$.

The simplex $F$ gets transformed under the variable change
$x_k = \dualpair{\omega_k, \theta}$ to a cusped region. For example in
case of the root system $A_2$ the fundamental region $F$ is an
equilateral triangle which gets transformed to a deltoid under
$(x_1,x_2)$. In Fig.~\ref{fig:Deltoid} the cusped regions for the
irreducible two-dimensional root systems are shown.
\begin{figure}
    \centering
    \includegraphics[width=0.3\textwidth]{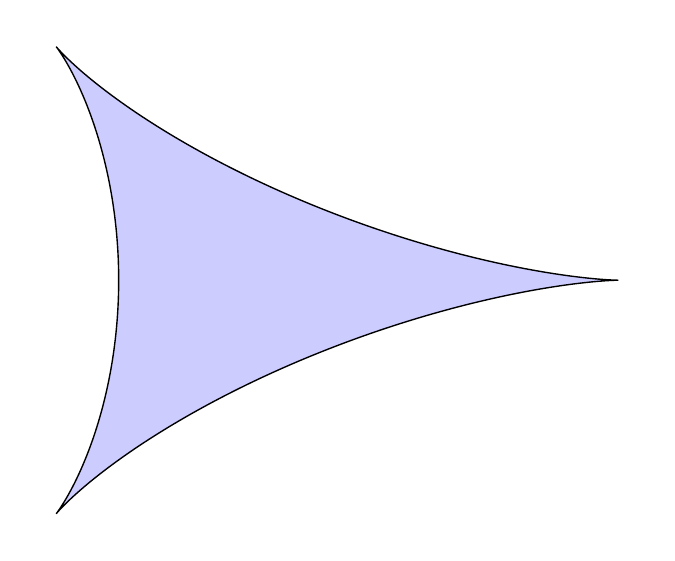}
    \includegraphics[width=0.3\textwidth]{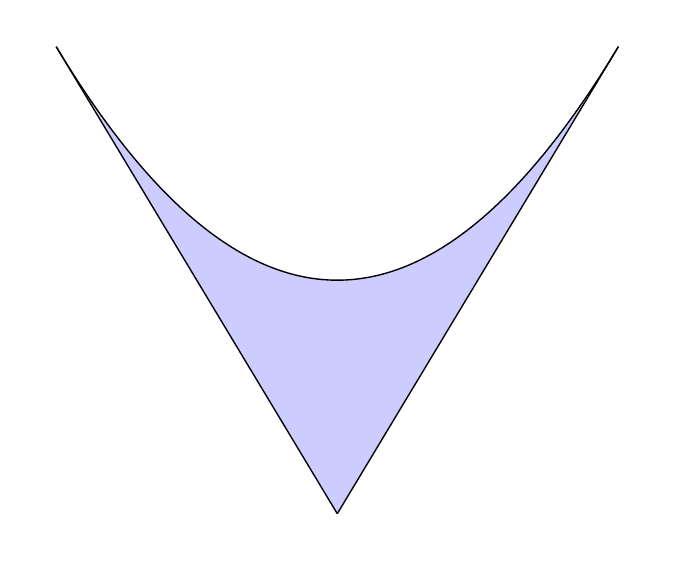}
    \includegraphics[width=0.3\textwidth]{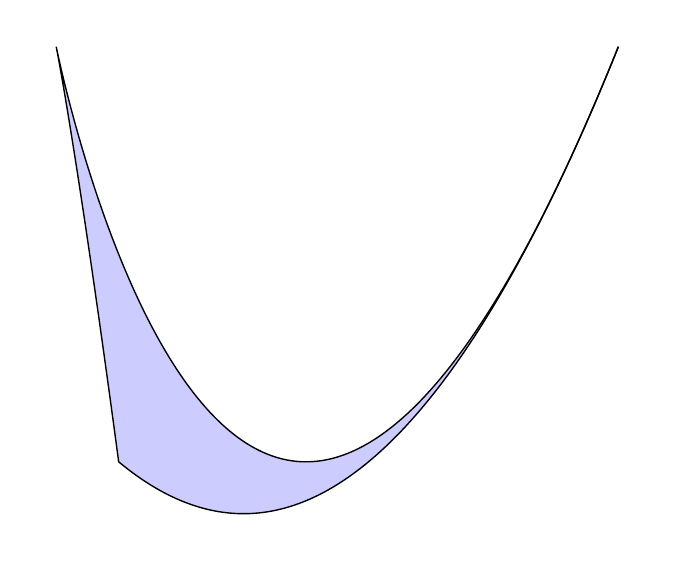}
    \caption{The image of the fundamental region $F$ under the
      variable change in case of $A_2$, $C_2$, and $G_2$,
      respectively.}
    \label{fig:Deltoid}
\end{figure}

The multivariate Chebyshev polynomials share many of the nice
properties the univariate ones posess. We list the properties used in
the sequel.
\begin{proposition}
    \label{proposition:PropertiesMultivariateChebyshev}%
    The multivariate Chebyshev polynomials associated to a Weyl group
    $W$ are subject to
    \begin{enumerate}[i.)]
        \item \label{item:InvarianceUnderAction}%
        invariance with respect to the action of the Weyl group on the
        weight indices
        \begin{equation}
            \label{eq:InvarianceWeylActionIndices}
            T_{w \lambda}  = T_{\lambda}
        \end{equation}
        and invariance with respect to the affine Weyl group on the
        argument in the fundamental domain
        \begin{equation}
            \label{eq:InvarianceActionArgument}
            T_{\lambda}( \dualpair{\omega_1,w \theta}, \dots,
            \dualpair{\omega_d,w \theta})
            = 
            T_{\lambda}( \dualpair{\omega_1,\theta}, \dots,
            \dualpair{\omega_d,\theta}),
        \end{equation}
        \item \label{item:ShiftProperty}%
        the shift property
        \begin{equation}
            \label{eq:ShiftProperty}
            T_{\lambda_1} T_{\lambda_2} = \tfrac{1}{|W|} \sum_{w \in
              W} T_{\lambda_1 + w \lambda_2},
        \end{equation}
        \item \label{item:DecompositionProperty}%
        the decomposition property
        \begin{equation}
            \label{eq:DecompositionProperty}
            T_{k \lambda}
            = T_{\lambda}(T_{k \omega_1},\dots, T_{k \omega_d})  
        \end{equation}
        for $\lambda \in P$ and $k \in \mathbb{Z}$.
    \end{enumerate}
\end{proposition}
The properties \ref{item:InvarianceUnderAction}.) and
\ref{item:ShiftProperty}.) of
Prop.~\ref{proposition:PropertiesMultivariateChebyshev} yield a
recursion relation if one uses the shift relation with the
$x_k = \dualpair{\omega_k, \theta}$.

The standard grading on the Chebyshev polynomials is insufficient for
our purposes as then, except for the case $A_n$, we do not obtain the
Gröbner basis property for $T_{n \omega_1}, \dots, T_{n \omega_d}$.
Replacing the standard degree with the m-degree introduced in
\cite{Moody.Patera:2011} solves this issue. The m-degree weights the
weights with the comarks.
\begin{definition}
    \label{definition:MDegree}%
    Let $\lambda \in P$ then its m-degree is
    \begin{equation}
        \label{eq:MDegree}
        \mdeg(\lambda) = \SP{\lambda, \alpha_0^\vee}.
    \end{equation}
    A monomial $x_1^{\lambda_1} \dots x_d^{\lambda_d}$ is then of
    m-degree $\mdeg(\lambda) = \mdeg((\lambda_1,\dots,\lambda_d))$.

    The \emph{$\mathrm{m}$-graded lexicographical ordering} on the monomials is
    defined by ordering the monomials with respect to $\mdeg$ and then
    breaking ties by the lexicographical order on the variables.
\end{definition}
In case of type $A_n$ all marks are equal to $1$, so in these cases
the m-degree coincides with the standard degree. The leading
monomial of the Chebyshev polynomial $T_\lambda$ with respect to the
m-graded lexicographical ordering is
$x_1^{\lambda_1} \dots x_d^{\lambda_d}$. By the recursion relations
obtained from the shift relation the leading monomials with respect to
m-graded lexicographical ordering of the polynomials
$T_{n \omega_1}, \dots, T_{n \omega_d}$ are disjoint. Hence they form
a Gröbner basis for the ideal they generate.
\begin{example}
    \label{example:ChebyshevCooleyTukeyA2}%
    The proposed algorithm for $A_2$, based on
    Theorem~\ref{theorem:CooleyTukeyTypeAlgorithms}, is an alternative
    to the one proposed in \cite{Pueschel.Roetteler:2008}, which
    relied on a version of
    Theorem~\ref{theorem:CooleyTukeyFFTUsingDecomposition}. The shift
    relation reads in this case
    \begin{equation}
        \label{eq:ShiftRelationA2}
        \begin{split}
            x_1 T_{\lambda_1, \lambda_2}
            &= \tfrac{1}{3} \left( T_{\lambda_1 + 1, \lambda_2} +
                T_{\lambda_1, \lambda_2 -1} + T_{\lambda_{1} - 1,
                  \lambda_2 + 1} \right) \\
            x_2 T_{\lambda_1, \lambda_2}
            &= \tfrac{1}{3} \left( T_{\lambda_1, \lambda_2 + 1} +
                T_{\lambda_1 - 1, \lambda_2} + T_{\lambda_1 + 1,
                  \lambda_2 - 1} \right).
        \end{split}
    \end{equation}
    A set of sufficient starting conditions for running the
    recursion is
    \begin{equation}
        \label{eq:StartingConditionsA2}
        \begin{split}
            T_{0,0} &= 1, \\
            T_{1,0} &= x, \\
            T_{0,1} &= y, \\
            T_{1,1} &= \tfrac{3}{2} x y - \tfrac{1}{2}.
        \end{split}
    \end{equation}
    We consider the signal model consisting of
    \begin{equation}
        \label{eq:A2ChebyshevSignalModel}
        \begin{split}
            \algebra{A}
            &= \mathbb{C}[x,y] \big/
            \ideal{T_{0,n}, T_{n,0}}, \\
            M
            &= \algebra{A} \text{ (as regular module)}, \\
        \end{split}
    \end{equation}
    and
    \begin{equation}
        \label{eq:A2ChebyshevSignalModelZTransform}
        \begin{split}
            \Phi \colon \mathbb{C}^{n^2}
            &\longrightarrow M, \\
            \Phi(s)
            &= \sum_{k ,\ell =0}^n s_{k,\ell}
            T_{k,\ell}. 
        \end{split}
    \end{equation}
    The visualization unveils the hexagonal lattice underlying this
    signal model, cf. Fig.~\ref{fig:UndirectedHexagonalLatticeA2},
    obtained from the shift relations~\eqref{eq:ShiftRelationA2}.
    \begin{figure}
        \centering
        \includegraphics[width=0.45\textwidth]{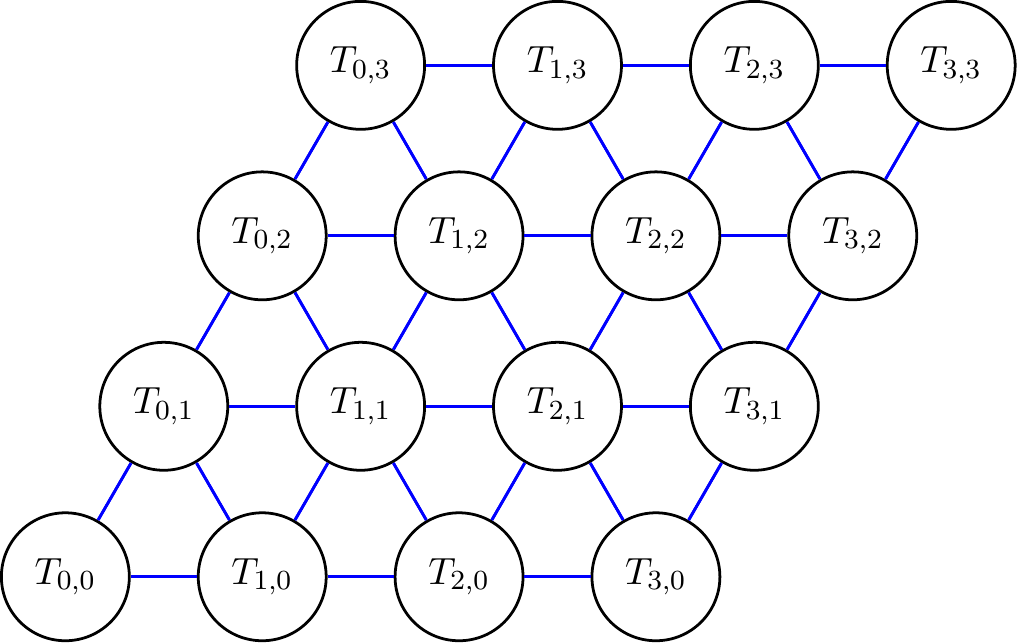}
        \includegraphics[width=0.45\textwidth]{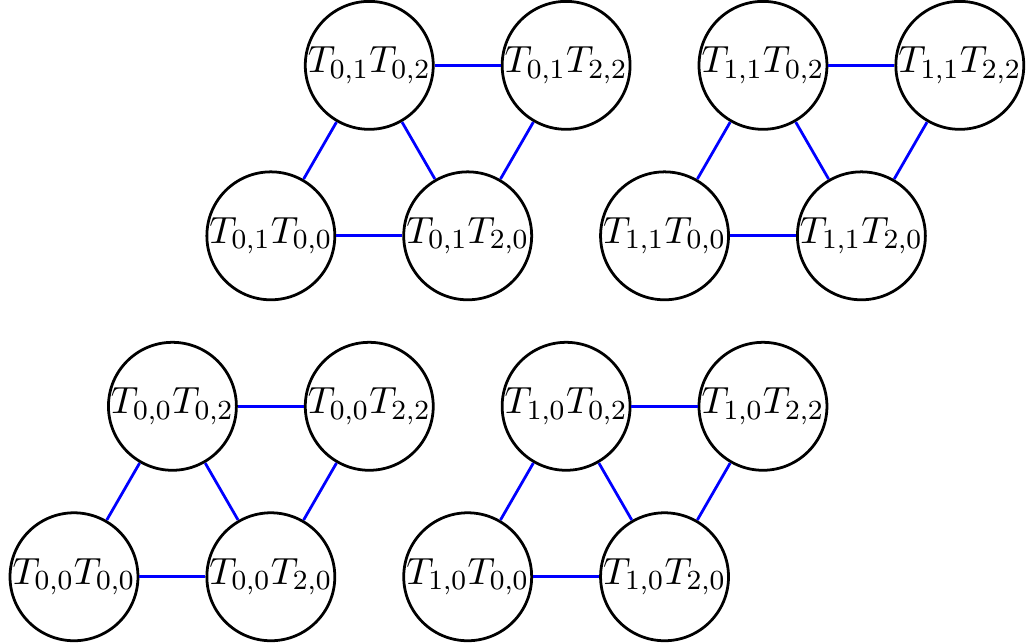}  
        \caption{Visualization of the signal model for Chebyshev
          polynomials of type $A_2$, a directed hexagonal lattice
          (left), and after representing the module as induction
          (right).}
        \label{fig:UndirectedHexagonalLatticeA2}
    \end{figure}
    In \cite{Pueschel.Roetteler:2008} the $n^2$ common zeros of
    $T_{n,0}$ an $T_{0,n}$ were described elementary. We propose a
    geometric description of the common zeros, as this shows the
    geometric mechanisms underlying the decomposition more clearly.
    The preimage of $0$ in the $x-y$-Domain is
    $\frac{1}{3} \omega_1^\vee + \frac{1}{3} \omega_2^\vee$. Through
    the stretching-and-folding property and the condition that the
    common zeros be in the fundamental domain $F$ one obtains
    \begin{equation}
        \label{eq:CommonZerosA2CoweightDescription}
        \begin{split}
            \Variety(\ideal{T_{n,0},T_{0,n}})
            &= \{ \frac{1 + 3 j }{3 n} \omega_1^\vee + \frac{1 + 3
              k}{3 n} \omega_2^\vee \; | \; 2 + 3( j + k) < 3 n\} \\
            &\quad \cup \{ \frac{2 + 3 j}{3 n} \omega_1^\vee +
            \frac{2 + 3 k}{3 n} \omega_2^\vee \; | \; 4 + 3 (j +
            k) < 3 n \},  
        \end{split}
    \end{equation}
    with $ j,k = 0,\dots,n-1$. For each
    $\alpha \in \Variety(\ideal{T_{r,0},T_{0,r}})$ one has
    $|\Variety(\ideal{T_{m,0} - \alpha_1, T_{0,m} - \alpha_2})| =
    m^2$. The geometric mechanism of the distribution of the common
    zeros is illustrated in Fig.~\ref{fig:SkewZerosOfA2}.
    \begin{figure}
        \centering
        \includegraphics[width=0.85\textwidth]{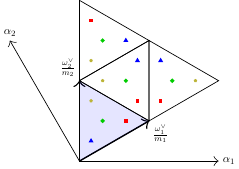}
        \caption{The four classes of common zeros of for the skew
          transforms in case $n = 2 \cdot 2$. The common zeros of
          $\ideal{T_{4,0}, T_{0,4}}$ are shown after being stretched
          by a factor of $2$. The action of the affine Weyl group,
          folding the stretched triangle back to the fundamental
          domain, is indicated. The different colors indicate which
          common zeros of $\ideal{T_{4,0}, T_{0,4}}$ are common zeros
          of which $\ideal{T_{2,0} - \alpha_1, T_{0,2} - \alpha_2}$.}
        \label{fig:SkewZerosOfA2}
    \end{figure}
    Since for $n = r \cdot m$ one thus has
    $|\Variety(\ideal{T_{n,0},T_{0,n}})| = |\Variety(\ideal{T_{r,0},
      T_{0,r}})| \cdot |\Variety(\ideal{T_{m,0},T_{0,m}})|$ and all
    the subalgebras
    $\mathbb{C}[x,y] \big/ \ideal{T_{m,0} - \alpha_1, T_{0,m} -
      \alpha_2}$ are of equal dimension, by
    Prop.~\ref{proposition:ExistenceOfTransversalByDecomposition} any
    basis of $\mathbb{C}[x,y] \big/ \ideal{T_{r,0}, T_{0,r}}$ is a
    transversal of $\mathbb{C}[x,y] \big/ \ideal{T_{m,0}, T_{0,m}}$.
    Since the basis change between the basis
    $(T_{0,0},\dots, T_{n,n})$ and the induction basis is sparse, see
    App.~\ref{subsec:A2TrafoBasisChange} for the concrete form, by
    Prop.~\ref{proposition:ComputationalCostCooleyTukey} the
    Theorem~\ref{theorem:CooleyTukeyTypeAlgorithms} yields a
    $O(n^2 \log(n))$ algorithm. This is substantially faster than the
    naive $O(n^4)$-approach.
\end{example}
\begin{example}
    \label{example:ChebyshevCooleyTukeyC2}%
    In case $C_2$ the shift relation is
    \begin{equation}
        \label{eq:RecurrenceRelationChebyshevC2}
        \begin{split}
            x_1 \cdot T_{k,\ell} &= \tfrac{1}{4} \left(
                T_{k+1,\ell} + T_{k-1,\ell} + T_{k-1,\ell+2} +
                T_{k+1,\ell-2} \right), \\ 
            x_2 \cdot T_{k,\ell} &= \tfrac{1}{4} \left(
                T_{k,\ell+1} + T_{k,\ell-1} + T_{k-1,\ell+1} +
                T_{k+1,\ell-1} \right). 
        \end{split}
    \end{equation}
    A set of sufficient starting conditions for running the
    recurrence relation is
    \begin{equation}
        \label{eq:StartingConditionsC2}
        \begin{split}
            T_{0,0} &= 1, \\
            T_{1,0} &= x_1, \\
            T_{0,1} &= x_2, \\
            T_{1,1} &= 2 x_1 x_2 - x_1.                 
        \end{split}
    \end{equation}
    The weight vector for the total $m$-degree lexicographical
    ordering of the monomials is $(1,2)$. That is $\mdeg(x_1) = 1$ and
    $\mdeg(x_2) = 2$.

    Denote by $\mathbb{T}_n = \{T_{k,\ell} \; | \; k+\ell=n\}$. Then
    one has a three-term recurrence of the form
    \begin{equation}
        \label{eq:ThreeTermRecurrenceRelationC2}
        x_i \mathbb{T}_k = A_{k,i} \mathbb{T}_{k+1} + B_{k,i} \mathbb{T}_k
        + C_{k,i} \mathbb{T}_{k-1},
    \end{equation}
    were the matrices $A_{k,i}, B_{k,i},$ and $C_{k,i}$ can be deduced
    from the shift relations~\eqref{eq:RecurrenceRelationChebyshevC2}. For
    example for the $x_1$ shift one obtains        
    \begin{equation}
        \label{eq:ThreeTermRecurrenceRelationMatricesA}
        A_{k,1} =
        \begin{bmatrix}
            0 & 1/2 & 0 & \dots & & 0 \\
            1/4 & 0 & 1/4 & 0 & \dots & 0 \\
            0 & \ddots & \ddots & \ddots &  & \vdots \\
            \dots & 0 & 1/4 & 0 & 1/4 & 0 \\
            0 & \dots & 0 & 1/2 & 0 & 1/4 
        \end{bmatrix}, 
    \end{equation}
    \begin{equation}
        \label{eq:ThreeTermRecurrenceRelationMatricesB}
        B_{k,1} =
        \begin{bmatrix}
            0 & \dots &  & 0 \\
            \vdots & \ddots & &\vdots \\
            0 & \dots & 0 & 0 \\
            0 & \dots & 1/4 & 0 \\
            0 & \dots & 0 & 0
        \end{bmatrix}, 
    \end{equation}
    and
    \begin{equation}
        \label{eq:ThreeTermRecurrenceRelationMatricesC}
        C_{k,1} =
        \begin{bmatrix}
            0 & 1/2 & 0 & \dots & 0 \\
            1/4 & 0 & 1/4 & \ddots & \vdots  \\
            0 & \ddots & \ddots & \ddots &   0 \\
            & & 1/4 & 0 & 1/4 \\
            \vdots & & 0 & 1/4 & 0 \\
            0 & \dots & & 0 & 1/4 
        \end{bmatrix},
    \end{equation}
    with special case $B_{1,1} =
    \begin{bsmallmatrix}
        1/2 & 0 \\
        0 & 0 
    \end{bsmallmatrix}$. The Christoffel-Darboux
    formula~\eqref{eq:ChristoffelDarboux} can be realized using the
    matrices $H_0 = \tfrac{1}{2}$ and
    $H_k = \mathsf{diag}(\tfrac{1}{8}, \tfrac{1}{16}, \dots,
    \tfrac{1}{16}, \tfrac{1}{8})$.

    We consider the signal model consisting of
    \begin{equation}
        \label{eq:C2ChebyshevSignalModel}
        \begin{split}
            \algebra{A}
            &= \mathbb{R}[x_1,x_2] \big/ \ideal{\mathbb{T}_n}, \\
            M
            &= \algebra{A} \text{ (as regular module)}, \\       
        \end{split}
    \end{equation}
    and
    \begin{equation}
        \label{eq:C2ChebyshevSignalModelZTransform}
        \begin{split}
            \Phi \colon \mathbb{R}^{\frac{n(n+1)}{2}}
            &\longrightarrow M, \\
            \Phi(s)
            &= \sum_{k + \ell < \frac{n (n+1)}{2}} s_{k,\ell}
            T_{k,\ell}. 
        \end{split}
    \end{equation}
    The signal model has a visualization, which resembles a
    triangle, cf. Fig.~\ref{fig:UndirectedC2Lattice}.
    \begin{figure}
        \centering
        \includegraphics[width=0.45\textwidth]{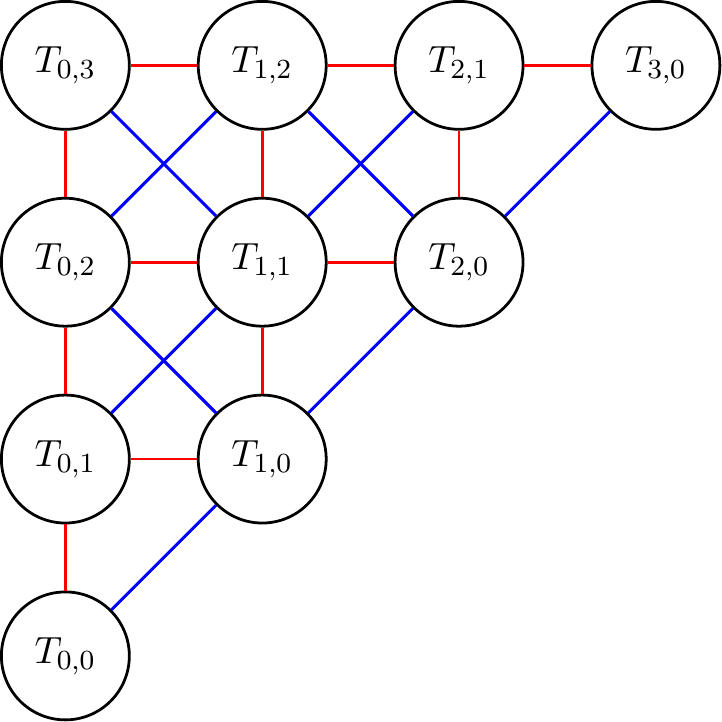}
        \hspace{0.05\textwidth}
        \includegraphics[width=0.45\textwidth]{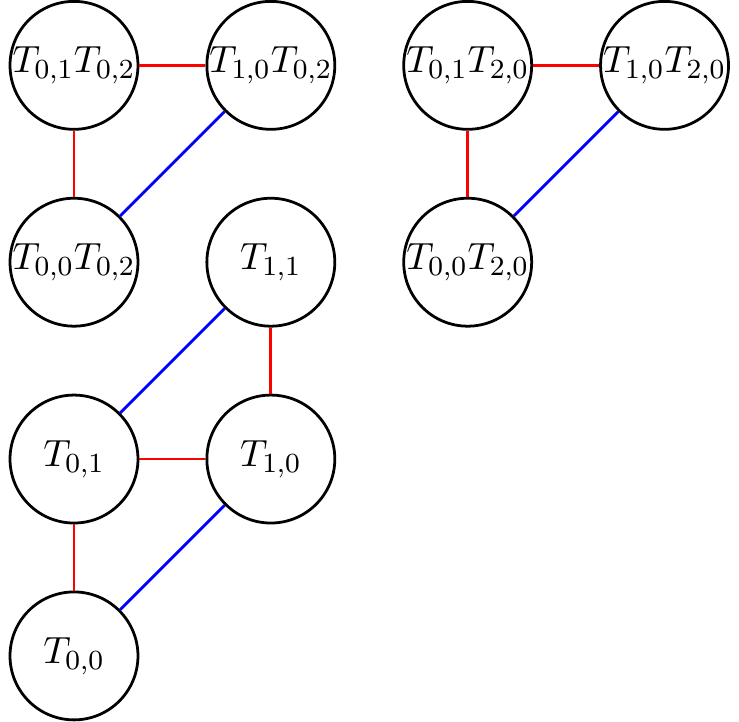} 
        \caption{Visualization of the signal model for Chebyshev
          polynomials of type $C_2$ on the left. The shifts of $x_1$
          are blue and the shifts of $x_2$ are red colored. On the
          right the decomposed lattice after the basis change is
          shown.}
        \label{fig:UndirectedC2Lattice}
    \end{figure}
    In \cite{Seifert.Hueper:2018b} we used an elementary description
    of the common zeros of $\mathbb{T}_n$. Here we present again a
    more geometric point of view, using the coweights. That is the
    common zeros are given as
    \begin{equation}
        \label{eq:CommonZerosC2Coweights}
        \Variety(\ideal{\mathbb{T}_n})
        = \Big\{ \tfrac{2 j + 1}{2 n} \omega_1^\vee + \tfrac{k}{2 n}
        \omega_2^\vee
        \; \Big| \; j,k = 0,\dots,n-1, \; j + k < n \}. 
    \end{equation}
    This results in $\frac{n (n+1)}{2}$ common zeros.

    We derive a fast algorithm in case $n = 2 \cdot m$. Since for
    $n = 2 \cdot m$ it is
    $\frac{n (n+1)}{2} \not= \frac{2 (2+1)}{2} \cdot \frac{m(m+1)}{2}$
    one does not get an induction via the decomposition.

    Due to the decomposition
    property~\ref{proposition:PropertiesMultivariateChebyshev},
    \ref{item:DecompositionProperty}.), the map
    \begin{equation}
        \label{eq:VarietyMappingC2}
        (x_1,x_2) \mapsto (T_{2,0}(x_1,x_2), T_{0,2}(x_1,x_2))
    \end{equation}
    maps the variety $\Variety(\mathbb{T}_n)$ to the variety
    $\Variety(\mathbb{T}_m)$. The map
    \begin{equation}
        \label{eq:ChebyshevStretchingMapC2}
        (x_1,x_2) \mapsto (T_{m,0}(x_1,x_2), T_{0,m}(x_1,x_2) 
    \end{equation}
    stretches the fundamental region $F$ by a factor of $m$ and folds
    it back under the affine Weyl group. After the stretching
    operation one obtains $m^2$ copies of the fundamental region. Thus
    if $\left(\frac{k}{2n}, \frac{j}{4n}\right)$ is in the interior of
    $F$ one obtains $m^2$ common zeros. If
    $\left(\frac{k}{2n}, \frac{j}{4n}\right)$ is on the boundary of
    $F$, that is $k = 0$, always two of the copies of $F$ in the
    interior of the stretched region share these common zeros. Hence
    in this case one only obtains $\frac{m(m+1)}{2}$ common zeros. Of
    the common zeros of $\mathbb{T}_2$ there is one in the interior
    and two on the boundary. Since these are the images under the
    stretching and folding operation one obtains two subalgebras with
    $\frac{m(m+1)}{2}$ common zeros and one subalgebra with $m^2$
    common zeros.
    \begin{figure}
        \centering
        \includegraphics[width=0.95\textwidth]{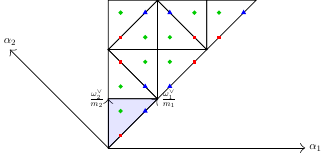}
        \caption{The three classes of common zeros of for the skew
          transforms in case $n = 2 \cdot 3$. The common zeros of
          $\mathbb{T}_6$ are shown after being stretched by a factor
          of $3$.The action of the affine Weyl group, folding the
          stretched triangle back to the fundamental domain, is
          indicated. The different colors indicate which common zeros
          of $\ideal{T_{6,0}, T_{0,6}}$ are common zeros of which
          $\ideal{T_{3,0} - \alpha_1, T_{0,3} - \alpha_2}$.}
        \label{fig:SkewZerosOfC2}
    \end{figure}
    Since the basis change obtained in
    App.~\ref{subsec:C2TrafoBasisChange} is sparse, we obtain a fast
    $O(n^2 \log n)$ algorithm by
    Theorem~\ref{theorem:CooleyTukeyFFTUsingDecomposition}.

    Since the number of common zeros of $\mathbb{T}_n$ equals
    $\dim \polynomials[2]{n-1}$
    Theorem~\ref{theorem:EquivalenceGaussianCubatureAndOrthogonalTransforms}
    implies the existence of an orthogonal transform. Denote the
    diagonal matrix with inverted entries by
    \begin{equation}
        \label{eq:DiagonalMatrixChristoffelEntries}
        D_n = \mathsf{diag}\left(1 \Big/  \left(\mathbb{T}_{n-1}^\top(x)
                H_{n-1}^{-1} A_{n-1,1} \tfrac{\partial}{\partial x_1}
                \mathbb{T}_{n}(x)\right) \right),
    \end{equation}
    and by
    \begin{equation}
        \label{eq:DirectSumHmatrices}
        H_n^{\oplus} = \bigoplus_{k=0}^{n-1} H_k^{-1}
    \end{equation}
    the direct sum of the $H_k^{-1}$ matrices. Reasoning analogously
    to the proof of
    Theorem~\ref{theorem:EquivalenceGaussianCubatureAndOrthogonalTransforms}
    an orthogonal version of the transform is given by
    \begin{equation}
        \label{eq:OrthogonalTrafoC2}
        \Fourier^{\mathsf{orth}} = \sqrt{H_n^{\oplus}} \cdot \Fourier
        \cdot \sqrt{D_n}.
    \end{equation}
    The matrix $H_n^{\oplus}$ is needed here since we do not have
    orthonormal but only orthogonal polynomials.
\end{example}

\begin{acknowledgements}
    This paper is part of the authors PhD thesis under the kind
    supervision of Knut Hüper. The author is grateful to Knut Hüper
    and Christian Uhl for helpful comments on draft versions of this
    paper. The author likes to thank Hans Munthe-Kaas for fruitful
    discussions about multivariate Chebyshev polynomials and the nice
    hospitality during a stay in Bergen.
\end{acknowledgements}

\appendix

\section{Basis changes}
\label{sec:BasisChanges}%

This appendix contains the basis changes for the
examples~\ref{example:ChebyshevCooleyTukeyA2} and
\ref{example:ChebyshevCooleyTukeyC2} of fast Chebyshev transforms with
$n = 2 m$. The Mathematica notebooks available at 
\url{https://github.com/bseifert-HSA/basis-change-Chebyshev-transforms},
which were used to calculate these basis changes, can be
modified to $n = r m$ for any integer $r$ but then the case analysis
becomes even longer. Hence we decided to put only the case $r = 2$
into writing.

Note that $T_{k \cdot m e_i} = T_{k e_i}(T_{m,0},T_{0,m})$ and in this
way is a sum of basis elements of the new basis and some parts which
reduce to a polynomial in the skew zeros.

\subsection{$A_2$ transform}
\label{subsec:A2TrafoBasisChange}%

The basis change is from $(T_{0,0}, \dots, T_{n,n})$ to
$(T_{0,0} T_{0,0} (T_{m,0}, T_{0,m}),\dots,\allowbreak T_{m-1,m-1}
T_{1,1} (T_{m,0},T_{0,m})$. For this one as to distinguish between
four regions of the indices: $(0,0) -- (m,m)$, $(m+1,0) -- (2 m, m)$,
$(0,m+1) -- (m, 2 m)$, and $(m+1,m+1) -- (2 m,2 m)$. Let
$k, \ell < m$.

The orbit of $(k,\ell)$ under the Weyl group $W(A_2)$ is
\begin{equation*}
    \{(-k-\ell,,k), (-k,k+\ell), (-\ell,-k), (\ell,-k-\ell), (k,\ell),
    (k+\ell,-\ell)\}. 
\end{equation*}
Hence one has to distinguish between the cases were $k = 0$,
$\ell = 0$, $k+\ell < m$, $k+\ell = m$, and $k+\ell > m$.

Region I:
\begin{equation*}
    T_{k,\ell} = T_{k,\ell}
\end{equation*}
Region II:
\begin{equation*}
    T_{m+k,\ell} =
    \begin{dcases}
        T_{m,0} & k,\ell = 0 \\
        - 2 T_{m-k,k} + 3 T_{k,0} T_{m,0} & \ell = 0 \\
        -\tfrac{1}{2} T_{m-\ell,0} + \tfrac{3}{2} T_{0,\ell} T_{m,0} &
        k = 0 \\
        - T_{m - k, k + \ell} - T_{m - k - \ell, k} + 3 T_{k,\ell}
        T_{m,0} & k + \ell < m \\
        - T_{0,k} + \tfrac{1}{2} T_{0,m-\ell} - \tfrac{3}{2}
        T_{\ell,0} T_{0,m} + 3 T_{k,\ell} T_{m,0} & k + \ell = m \\
        T_{\ell, 2m - k - \ell} - 3 T_{m-k, k + \ell - m} T_{0,m} + 3
        T_{k,\ell} T_{m,0} & k + \ell > m 
    \end{dcases}    
\end{equation*}
Region III:
\begin{equation*}
    T_{k,m+\ell}
    =
    \begin{dcases}
        T_{0,m} & k,\ell = 0 \\
        - \tfrac{1}{2} T_{0,m-k} + \tfrac{3}{2} T_{k,0} T_{0,m} & \ell
        = 0 \\
        2 T_{\ell,m-\ell} + 3 T_{0,\ell} T_{0,m} & k = 0 \\
        - T_{\ell, m -k -\ell} - T_{k+\ell, m-\ell} + 3 T_{k,\ell}
        T_{0,m} & k+l < m \\
        -\tfrac{1}{2} T_{\ell,0} - \tfrac{3}{2} T_{0,-m-\ell} T_{m,0}
        + 3 T_{k,\ell} T_{0,m} & k+\ell = m \\
        T_{2m - k -\ell,k} - 3 T_{k+\ell-m,m-\ell} T_{m,0} + 3
        T_{k,\ell} T_{0,m} & k+\ell > m 
    \end{dcases}
\end{equation*}
Region IV:
\begin{equation*}
    T_{m+k,m+\ell} = 
    \begin{dcases}
        T_{m,m} & k,\ell = 0 \\
        T_{k,0} - 3 T_{m-k,k} T_{0,m} + 3 T_{k,0} T_{m,m} & \ell = 0
        \\ 
        T_{0,\ell} - 3 T_{\ell,m-\ell} T_{m,0} + 3 T_{0,\ell} T_{m,m}
        & k = 0 \\ 
        \begin{split}
        & 2 T_{k,\ell} - T_{m-\ell, m-k} - 3 T_{m-k,k+\ell} T_{0,m}
        \ldots \\
        &- 3 T_{k+\ell, m-\ell} T_{m,0} + 6 T_{k,\ell} T_{m,m}            
        \end{split} & k+\ell < m
        \\ 
        \begin{split}
            &T_{k,\ell} - \tfrac{3}{2} T_{0,k} T_{0,m} - \tfrac{3}{2}
            T_{\ell,0} T_{0,2m} \ldots \\
            & - \tfrac{3}{2} T_{\ell,0} T_{m,0} -
            \tfrac{3}{2} T_{0,m-\ell} T_{2m,0} + 6 T_{k,\ell} T_{m,m}             
        \end{split} &
        k+\ell = m \\
        \begin{split}
        &- T_{m-\ell,m-k} + 2 T_{k,\ell} - 3 T_{m-k, k+\ell-m} T_{m,0} 
        \ldots \\
        &+ 3 T_{\ell, 2m - k - \ell} T_{0,m} - 3 T_{m-k, k+ \ell - m}
        T_{0,2m}  \ldots \\
        &+ 3 T_{2m - k - \ell,k} T_{m,0} - 3 T_{k+\ell-m,m-\ell}
        T_{2m,0}  \ldots \\
        &  - 3 T_{k+\ell-m,m-\ell} T_{0,m} + 6  T_{k,\ell} T_{m,m}
        \end{split} & k+\ell > m 
    \end{dcases}
\end{equation*}

\subsection{$C_2$ transform}
\label{subsec:C2TrafoBasisChange}%

The basis change is from $(T_{k,\ell} \; | \; k + \ell < n)$ to
$(T_{k,\ell} T_{t,p}(T_{m,0},T_{0,m}) \; | \; k + \ell < m, t+p < 2)
$. For this one as to distinguish between three regions of the
indices: $(0,0) -- (m,m)$, $(m+1,0) -- (2 m, m)$, and
$(0,m+1) -- (m, 2 m)$. Let $k, \ell < m$.

The orbit of $(k,\ell)$ under the Weyl group $W(C_2)$ is
\begin{equation*}
    \{(-k-2 \ell,\ell), (-k-2 \ell,k+\ell), (-k,-\ell), (-k,k+\ell),
    (k,-k-\ell), {k,\ell}, (k+2 \ell,-k-\ell), (k+2 \ell,-\ell) \} 
\end{equation*}
Hence one has to distinguish the cases $\ell = 0$,
$k = 0 \wedge 2 \ell < m$, $k = 0 \wedge 2 \ell = m$,
$k = 0 \wedge 2 \ell > m$, $k + 2 \ell < m$, $k + 2 \ell = m$,
$k + 2 \ell > m$. Note that since for any basis elements index $(t,p)$
one has $t+p < n = 2 m$ one always has $k+\ell < m$ in the sequel.

Region I:
\begin{equation*}
    T_{k,\ell} = T_{k,\ell}
\end{equation*}
Region II:
\begin{equation*}
    T_{m+k,\ell} =
    \begin{dcases}
        T_{m,0} & k,\ell = 0 \\
        -2 T_{-k + m,k} - T_{-k + m,0} + 4 T_{k,0} \cdot T_{m,0} & \ell = 0\\
        -\tfrac{1}{2} T_{-2 \cdot\ell + m,\ell} + 2 T_{0,\ell} \cdot
        T_{m,0}& k = 0, 2 \ell < m \\
        -\tfrac{1}{2} T_{0,\ell} + 2 T_{0,\ell} \cdot T_{m,0}& k = 0,
        2 \ell = m \\
        -\tfrac{1}{2} T_{2 \cdot \ell - m,-\ell + m} + 2 T_{0,\ell}
        \cdot T_{m,0} &  k = 0, 2 \ell > m \\
        \begin{split}
        &- T_{-k + m,k + \ell} - T_{-k - 2 \cdot \ell + m,\ell} \ldots
        \\
        &- T_{-k - 2 \cdot \ell + m,k + \ell} + 4 T_{k,\ell} \cdot
        T_{m,0}             
        \end{split} &
         k + 2 \ell < m \\
         - T_{0,\ell} - T_{0,k + \ell} - T_{2 \cdot \ell,k + \ell} + 4
         T_{k,\ell} \cdot T_{m,0}
         & k + 2 \ell = m \\
         \begin{split}
             & - T_{k + 2 \cdot \ell - m,-\ell + m}
             - T_{k + 2 \cdot \ell - m,-k - \ell + m} \ldots \\
             & - T_{-k + m,k + \ell} 
             + 4 T_{k,\ell} \cdot T_{m,0} 
         \end{split}
         & k + 2 \ell > m
    \end{dcases}
\end{equation*}
Region III:
\begin{equation*}
    T_{k,m+\ell} =
    \begin{dcases}
        T_{0,m} & k,\ell = 0 \\
        - T_{k,-k + m} + 2 T_{k,0} \cdot T_{0,m} & \ell = 0\\
        - T_{0,-\ell + m} - T_{2 \cdot \ell,-\ell + m} + 4 T_{0,\ell}
        \cdot T_{0,m}
        & k = 0, 2 \ell < m \\
        - \frac{1}{2} T_{0,\ell} - T_{0,\ell} T_{m,0} + 4 T_{0,\ell}
        \cdot T_{0,m}
        & k = 0, 2 \ell = m \\
        + T_{0,\ell} + T_{-2 \cdot \ell + 2 \cdot m,\ell} -4 T_{2
          \cdot \ell - m,-\ell + m} \cdot T_{m,0} + 4 T_{0,\ell} \cdot
        T_{0,m}
        & k = 0, 2 \ell > m \\
        - T_{k,-k - \ell + m} - T_{k + 2 \cdot \ell,-\ell + m} - T_{k
          + 2 \cdot \ell,-k - \ell + m} + 4 T_{k,\ell} \cdot T_{0,m}
        & k + 2 \ell < m \\
        -2 T_{0,\ell} \cdot T_{m,0} -2 T_{0,-\ell + m} \cdot T_{m,0} +
        4 T_{k,\ell} \cdot T_{0,m}
        & k + 2 \ell = m \\
        \begin{split}
            & + T_{k,-k - \ell + m}
            + 2 T_{k,\ell}
            + T_{-k - 2 \cdot \ell + 2 \cdot m,k + \ell} \ldots \\
            & -4 T_{k + 2 \cdot \ell - m,-\ell + m} \cdot T_{m,0}
            + T_{-k - 2 \cdot \ell + 2 \cdot m,\ell} \ldots \\
            & -4 T_{k + 2 \cdot \ell - m,-k - \ell + m} \cdot T_{m,0}
            + 4 T_{k,\ell} \cdot T_{0,m} 
        \end{split}
        & k + 2 \ell > m
    \end{dcases}
\end{equation*}

\clearpage
\bibliographystyle{spmpsci}      
\bibliography{Literatur}

\end{document}